\documentclass[11pt]{amsart}
\usepackage{amsfonts,amscd}
\usepackage{amssymb}
\usepackage{url}

\theoremstyle{plain}
\newtheorem{theorem}                {Theorem}      [section]
\newtheorem{proposition}  [theorem]  {Proposition}
\newtheorem{corollary}    [theorem]  {Corollary}
\newtheorem{lemma}        [theorem]  {Lemma}

\theoremstyle{definition}
\newtheorem{example}      [theorem]  {Example}
\newtheorem{remark}       [theorem]  {Remark}

\def \ec{\mathbb{E}_{\mathbb C}}

\def \r{{\mathbb R}}
\def \s{{\mathbb S}}

\def \cp{\mathbb CP}
\def \C{\mbox{${\mathbb C}$}}
\def \jc{\bar{J}}
\def \jr{\hat{J}}
\def \nablar{{\hat{\nabla}}}
\def \nablac{\nabla^{\ic}}
\def \nablas{\nabla^{\is}}
\def \is{\jmath}
\def \ic{\bar{\jmath}}
\def \bx{\bar{X}}
\def \bv{\bar{V}}
\def \be{\bar{E}}
\def \bgamma{\bar{\gamma}}
\def \nablab{\bar{\nabla}}
\def \bk{\bar{k}}

\def \a0{\alpha_0}
\def \bh{\bar{H}}
\def \bm{\bar{M}}

\DeclareMathOperator{\trace}{trace}
\DeclareMathOperator{\di}{div}
\DeclareMathOperator{\grad}{grad}

\DeclareMathOperator{\Div}{div}

\DeclareMathOperator{\cst}{constant}

\numberwithin{equation}{section}

\begin{document}

\title{Biharmonic submanifolds of $\cp^n$}

\author{D.~Fetcu}
\author{E.~Loubeau}
\author{S.~Montaldo}
\author{C.~Oniciuc}

\address{Department of Mathematics\\
"Gh. Asachi" Technical University of Iasi\\
Bd. Carol I no. 11 \\
700506 Iasi, Romania} \email{dfetcu@math.tuiasi.ro}

\address{D{\'e}partement de Math{\'e}matiques \\
Universit{\'e} de Bretagne Occidentale \\
6, av. Le Gorgeu \\
29238 Brest Cedex 3, France} \email{loubeau@univ-brest.fr}

\address{Universit\`a degli Studi di Cagliari\\
Dipartimento di Matematica e Informatica\\
Via Ospedale 72\\
09124 Cagliari, Italia}
\email{montaldo@unica.it}

\address{Faculty of Mathematics\\ ``Al.I. Cuza'' University of Iasi\\
Bd. Carol I no. 11 \\
700506 Iasi, Romania}
\email{oniciucc@uaic.ro}

\thanks{The third author was supported by PRIN-2005 (Italy):
Riemannian Metrics and Differentiable Manifolds. The last author was
partially supported by the Grant IDEI, no. 2228/2009, Romania}

\begin{abstract}
We give some general results on proper-biharmonic submanifolds of a
complex space form and, in particular, of the complex projective
space. These results are mainly concerned with submanifolds with
constant mean curvature or parallel mean curvature vector field. We
find the relation between the bitension field of the inclusion of a
submanifold $\bar{M}$ in $\cp^n$ and the  bitension field of the
inclusion of the corresponding  Hopf-tube in $\s^{2n+1}$. Using this
relation we produce new families of proper-biharmonic submanifolds
of $\cp^n$. We study the geometry of biharmonic curves of  $\cp^n$
and we characterize the proper-biharmonic curves in terms of their
curvatures and complex torsions.
\end{abstract}

\date{}

\subjclass[2000]{58E20}

\keywords{Harmonic maps, biharmonic maps, biharmonic submanifolds}

\maketitle

\section{Introduction}
{\it Biharmonic} maps $\varphi:(M,g)\to(N,h)$ between Riemannian
manifolds are critical points of the {\em bienergy} functional
$$
E_2(\varphi)=\frac{1}{2}\int_{M}\,|\tau(\varphi)|^2 \ v_g,
$$
where $\tau(\varphi)=\trace\nabla d\varphi$ is the {tension field}
of $\varphi$ that vanishes  on harmonic maps. The Euler-Lagrange
equation corresponding to $E_2$ is given by the vanishing of the
{\it bitension field}
\begin{equation}
\label{eq-tau2}
\tau_2(\varphi)=-J^{\varphi}(\tau(\varphi))=-\Delta^{\varphi}\tau(\varphi)
-\trace R^N(d\varphi,\tau(\varphi))d\varphi,
\end{equation}
where $J^{\varphi}$ is formally the Jacobi operator of $\varphi$
(see ~\cite{GYJ}). The operator $J^{\varphi}$ is linear, thus any
harmonic map is biharmonic. We call {\it proper-biharmonic} the
non-harmonic biharmonic maps.

The analytic aspects of biharmonic maps as well as the differential
geometry of such maps have been studied in the last decade (see, for
example, ~\cite{SYACLWPY,TL,RM1,RM2} and
~\cite{PBAFSO,YJCRAW,GYJ,SMCO,YLO1}, respectively).

In this paper we shall focus our attention on {\it proper-biharmonic
submanifolds}, i.e.  on submanifolds such that the inclusion map is
a proper-biharmonic map.

The proper-biharmonic submanifolds of a real space form were
extensively studied, see, for example,
~\cite{ABSMCO,ABSMCO1,RCSMCO1,RCSMCO2,BYC,ID}. Naturally, the next
step has been the study of proper-biharmonic submanifolds of spaces
of non-constant sectional curvature (see, for
example,~\cite{KARECMTS,DFCO,TIJIHU,JI,YLO2,TS,WZ}).

This work is devoted to the study of proper-biharmonic submanifolds
in a complex space form. This subject has already been started by
several authors. In ~\cite{YJCHS1} some pinching conditions for the
second fundamental form and the Ricci curvature of a biharmonic
Lagrangian submanifold of $\cp^n$, with parallel mean curvature
vector field, were obtained. In \cite{TS}, the author gave a
classification of biharmonic Lagrangian surfaces of constant mean
curvature in $\cp^2$. Finally, in \cite{TIJIHU}, there is a
characterization of biharmonic constant mean curvature real
hypersurfaces of $\cp^n$ and the classification of biharmonic
homogeneous real hypersurfaces of $\cp^n$.

The paper is organized as follows. In the first part we obtain some
general properties on proper-biharmonic submanifolds with constant
mean curvature, or parallel mean curvature vector field, of the
complex projective space endowed with the standard Fubini-Study
metric. When the ambient space is a complex space form of
non-positive holomorphic curvature we obtain non-existence results.

In the second part we consider the Hopf map defined as the
restriction of the natural projection
$\pi:\C^{n+1}\setminus\{0\}\to\cp^n$ to the sphere $\s^{2n+1}$,
which defines a Riemannian submersion. For a real submanifold
$\bar{M}$ of $\cp^n$ we denote by $M:=\pi^{-1}(\bar{M})$ the
Hopf-tube over $\bar{M}$. We obtain the formula which relates the
bitension field of the inclusion of $\bar{M}$ in $\cp^n$ and the
bitension field of the inclusion of $M=\pi^{-1}(\bar{M})$ in
$\s^{2n+1}$ (Theorem~\ref{teo:reltautau}). Using this formula we are
able to produce a new class of proper-biharmonic submanifolds
$\bar{M}$ of $\cp^n$ when $M$ is of ``Clifford type''
(Theorem~\ref{eq:Cliffordtype}), and to reobtain a result in
~\cite{WZ} when $M$ is a product of circles
(Theorem~\ref{eq:Zhangtheorem}).

\noindent We note that $\bar{M}$ is minimal (harmonic) in $\cp^n$ if
and only if $M$ is minimal in $\s^{2n+1}$ (see~\cite{HBL}) but, for
what concerns the biharmonicity, the result does not hold anymore.

In the last part of the paper we concentrate on the geometry of
proper-biharmonic curves of $\cp^n$. We characterize all
proper-biharmonic curves of $\cp^n$ in terms of their curvatures and
complex torsions. Then, using the classification of holomorphic
helices of $\cp^2$ given in ~\cite{SMTA}, we determine all
proper-biharmonic curves of $\cp^2$ (Theorem~\ref{eq:curvescp2}).

\section{Biharmonic submanifolds of complex space forms}

Let $\ec^n(4c)$ be a complex space form of holomorphic sectional
curvature $4c$. Let us denote by $\jc$  the complex structure and by
$\langle,\rangle$ the Riemannian metric on $\ec^n(4c)$. Then its
curvature operator is given, for vector fields $X,Y$ and $Z$, by
\begin{eqnarray}\label{eq:cpn-curvature}
R^{\ec^n(4c)}(X,Y)Z&=&{c}\{\langle Y,Z\rangle X-\langle X,Z\rangle Y\\
&&+\langle\jc Y,Z\rangle\jc X-\langle\jc X,Z\rangle\jc Y+2\langle
X,\jc Y\rangle\jc Z\}.\nonumber
\end{eqnarray}

\noindent Let now
$$
\ic:\bm^{\bar{m}}\to\ec^n(4c)
$$
be the canonical inclusion of a submanifold $\bm$ in  $\ec^n(4c)$ of
real dimension $\bar{m}$. Then the bitension field becomes
\begin{equation}\label{eq-tau2-csf}
\tau_2(\ic)=-\bar{m}\{\Delta^{\ic} {\bh}-{c\bar{m}}{\bh}+{3c}
\jc\left(\jc {\bh}\right)^{\top}\},
\end{equation}
where ${\bh}$ denotes the  mean curvature vector field,
$\Delta^{\ic}$ is the  rough Laplacian,  and $()^{\top}$ denotes the
tangential component to $\bm$. The overbar notation will be
justified in the next section. If we assume that $\jc\bar{H}$ is
tangent to $\bar{M}$, then \eqref{eq-tau2-csf} simplifies to
\begin{equation}\label{eq-tau2-jhtan}
\tau_2(\ic)=-\bar{m}\{\Delta^{\ic} {\bh}-{c}(\bar{m}+3){\bh}\}.
\end{equation}
Decomposing \eqref{eq-tau2-jhtan} with respect to its tangential and normal component we get

\begin{proposition}\label{pro:decomp-tau2}
Let $\bar{M}$ be a real submanifold of $\ec^n(4c)$ of dimension
$\bar{m}$ such that $\jc\bar{H}$ is tangent to $\bar{M}$. Then
$\bar{M}$ is biharmonic if and only if
\begin{equation}
\label{eq:caracterizarebiarmonicitate}
\begin{cases}
\Delta^{\perp}\bar{H}+\trace
\bar{B}(\cdot,\bar{A}_{\bar{H}}(\cdot))-c(\bar{m}+3)\bar{H}=0\\
4\trace\bar{A}_{\nabla^{\perp}_{(\cdot)}\bar{H}}(\cdot)+\bar{m}\grad
(\vert\bar{H}\vert^2)=0
\end{cases},
\end{equation}
where $\bar A$ denotes the Weingarten operator, $\bar B$ the second
fundamental form, $\bar H$ the mean curvature vector field,
$\nabla^\perp$ and $\Delta^\perp$ the connection and the Laplacian
in the normal bundle of $\bar M$ in $\ec^n(4c)$.
\end{proposition}

\begin{proof}
Since $\bar{H}$ is normal to $\bar{M}$, from \eqref{eq-tau2-jhtan}
we only have to split $\Delta^{\ic}{\bh}$. With respect to a
geodesic frame $\{X_i\}_{i=1}^{\bar{m}}$ around an arbitrary point
$p\in\bar{M}$, we have
$$
-\Delta^{\ic}{\bh}=\sum_{i=1}^{\bar{m}}\nabla_{X_i}^{\ic}
\nabla_{X_i}^{\ic}\bar{H}.
$$
Thus, around $p$,
$$
\nabla_{X_i}^{\ic}\bar{H}=\nabla_{X_i}^{\perp}\bar{H}-\bar{A}_{\bar{H}}(X_i),
$$
$$
\nabla_{X_i}^{\ic}\nabla_{X_i}^{\ic}\bar{H}=
\nabla_{X_i}^{\perp}\nabla_{X_i}^{\perp}\bar{H}-
\bar{A}_{\nabla_{X_i}^{\perp}\bar{H}}(X_{i})-\bar{B}(X_{i},\bar{A}_{\bar{H}}(X_i))
-\nabla^{\bar M}_{X_i}\bar{A}_{\bar{H}}(X_i)
$$
and, at $p$,
$$
\sum_{i=1}^{\bar{m}}\nabla_{X_i}^{\ic}\nabla_{X_i}^{\ic}\bar{H}=-\Delta^{\perp}\bar{H}
-\trace
\bar{A}_{\nabla_{(\cdot)}^{\perp}\bar{H}}(\cdot)-\trace\bar{B}(\cdot,
\bar{A}_{\bar{H}}(\cdot))-\trace\nabla^{\bar M}
\bar{A}_{\bar{H}}(\cdot,\cdot),
$$
where $\nabla^{\bar M}$ is the Levi-Civita connection on $\bar M$.
Moreover, a long but straightforward computation gives
\begin{eqnarray*}
\trace\nabla^{\bar M}\bar{A}_{\bar{H}}(\cdot,\cdot)&=&\sum_{i=1}^{\bar{m}}\nabla^{\bar M}_{X_{i}}\bar{A}_{\bar{H}}(X_{i})\\
&=& \sum_{i,j}\nabla^{\bar
M}_{X_i}(\langle\bar{A}_{\bar{H}}(X_i),X_j\rangle X_j)
=\sum_{i,j}(X_i\langle\bar{A}_{\bar{H}}(X_i),X_{j}\rangle)X_j\\&=&
\sum_{i,j}(X_i\langle\bar{B}(X_j,X_i),\bar{H}\rangle)X_j=
\sum_{i,j}(X_i\langle\nabla_{X_j}^{\ic}X_i,\bar{H}\rangle)X_j\\&=&
\sum_{i,j}\{
\langle\nabla_{X_i}^{\ic}\nabla_{X_j}^{\ic}X_i,\bar{H}\rangle+
\langle\nabla_{X_j}^{\ic}X_i,\nabla_{X_i}^{\ic}\bar{H}\rangle\}X_j\\&=&
\sum_{i,j}\{\langle
\nabla_{X_i}^{\ic}\nabla_{X_j}^{\ic}X_i,\bar{H}\rangle+
\langle\bar{B}(X_j,X_i),\nabla_{X_i}^{\perp}\bar{H}
\rangle\}X_j\\&=&\sum_{i,j}\{\langle
\nabla_{X_i}^{\ic}\nabla_{X_j}^{\ic}X_{i},\bar{H}\rangle+\langle
\bar{A}_{\nabla_{X_i}^{\perp}\bar{H}}(X_i),X_j\rangle\}X_j\\&=&
\sum_{i,j}\langle\nabla_{X_i}^{\ic}\nabla_{X_j}^{\ic}X_i,\bar{H}\rangle
X_j +\sum_i\bar{A}_{\nabla_{X_i}^{\perp}\bar{H}}(X_i).
\end{eqnarray*}
Further, using the curvature tensor field of the pull-back bundle
$(\ic)^{-1}T\ec^n(4c)$, we get
\begin{eqnarray*}
\trace\nabla^{\bar
M}\bar{A}_{\bar{H}}(\cdot,\cdot)&=&\sum_{i,j}\langle
R^{\ec^n(4c)}(X_i,X_j)X_i+\nabla_{X_j}^{\ic}\nabla_{X_i}^{\ic}X_i+\nabla_{[X_i,
X_j]}^{\ic}X_i,\bar{H}\rangle X_j\\
&&+
\sum_i\bar{A}_{\nabla_{X_i}^{\perp}\bar{H}}(X_i)\\&=&c\sum_{i,j}\langle
\langle X_i,X_j\rangle X_i-\langle X_i, X_i\rangle X_j\\
&& +\langle\jc X_j,X_i\rangle \jc X_i-\langle\jc
X_i,X_i\rangle\jc X_j+2\langle X_i,\jc X_j\rangle\jc X_i,
\bar{H}\rangle
X_j\\
&&+\sum_{i,j}\langle\nabla_{X_j}^{\ic}\bar{B}(X_i,X_i)+
\nabla_{X_j}^{\ic}\nabla_{X_i}X_{i},\bar{H}\rangle X_j+
\sum_i\bar{A}_{\nabla_{X_i}^{\perp}\bar{H}}(X_i)\\
&=&3c\sum_{i,j}\langle\jc(\langle\jc X_j,X_i\rangle X_i),\bar{H}
\rangle X_j+\bar{m}\sum_{j}\langle
\nabla_{X_j}^{\ic}\bar{H},\bar{H}\rangle X_j\\&&+\sum_{i,j}\langle
\nabla^{\bar M}_{X_j}\nabla^{\bar
M}_{X_i}X_i+\bar{B}(X_j,\nabla^{\bar M}_{X_i}X_i),\bar{H}\rangle
X_j+\sum_i\bar{A}_{\nabla_{X_i}^{\perp}\bar{H}}(X_i)\\
&=& \frac{\bar{m}}{2}\sum_jX_j(|\bar{H}|^2)X_j\\&&+
3c\sum_{j}\langle \jc ((\jc X_j)^{\top}),\bar{H}\rangle
X_j+\sum_i\bar{A}_{\nabla_{X_i}^{\perp}\bar{H}}(X_i).
\end{eqnarray*}
Therefore
\begin{eqnarray*}
\sum_{i=1}^{\bar{m}}\nabla^{\bar
M}_{X_{i}}\bar{A}_{\bar{H}}(X_{i})&=&
\frac{\bar{m}}{2}\grad(|\bar{H}|^2)+3c\sum_{j}\langle\jc((\jc
X_{j})^{\top}),\bar{H}\rangle X_j+
\sum_i\bar{A}_{\nabla_{X_i}^{\perp}\bar{H}}(X_{i}).
\end{eqnarray*}
Finally, taking into account that $\jc\bar{H}$ is tangent to $\bar{M}$,
we have
$$
\Delta^{\ic}\bar{H}=\Delta^{\perp}\bar{H} +2 \trace
\bar{A}_{\nabla_{(\cdot)}^{\perp}\bar{H}}(\cdot)+\trace\bar{B}(\cdot,
\bar{A}_{\bar{H}}(\cdot))+\frac{\bar{m}}{2}\grad(|\bar{H}|^2)
$$
which gives, together with \eqref{eq-tau2-jhtan}, the desired result.
\end{proof}

If $\bar{M}$ is a hypersurface, then $\jc\bar{H}$ is tangent to
$\bar{M}$, and the previous proposition gives the following result
of \cite{TIJIHU}

\begin{corollary}\label{cor:hyp}
Let $\bar{M}$ be a real hypersurface of $\ec^n(4c)$ of non-zero
constant mean curvature. Then it is proper-biharmonic if and only if
$$
\vert\bar{B}\vert^2=2c(n+1).
$$
\end{corollary}

Proposition~\ref{pro:decomp-tau2} can be applied also in the case of
Lagrangian submanifolds. We recall here that $\bar M$ is called a
Lagrangian submanifold if $\dim\bar M=n$ and $\ic^{\ast}\Omega=0$,
where $\Omega$ is the fundamental $2$-form on $\ec^n(4c)$ defined by
$\Omega(X,Y)=\langle X,\jc Y\rangle$, for any vector fields $X$ and
$Y$ tangent to $\ec^n(4c)$.

\begin{corollary}
Let $\bar{M}$ be a Lagrangian submanifold of $\ec^n(4c)$ with
parallel mean curvature vector field. Then it is biharmonic if and
only if
$$
\trace\bar{B}(\cdot,\bar{A}_{\bar{H}}(\cdot))=c(n+3)\bar{H}.
$$
\end{corollary}

In the sequel we shall consider only the case of complex space forms
with positive holomorphic sectional curvature. A partial motivation
of this fact is that Corollary~\ref{cor:hyp} rules out the case
$c\leq 0$. As usual, we consider the complex projective space
$\cp^n=(\C^{n+1}\setminus\{0\}) / \r^{\ast}$, endowed with the
Fubini-Study metric, as the model for the complex space form of
positive constant holomorphic sectional curvature $4$.

\begin{proposition}
Let $\bar{M}$ be a real submanifold of $\cp^n$ of dimension
$\bar{m}$ such that $\jc\bar{H}$ is tangent to $\bar{M}$. Assume
that it has non-zero constant mean curvature. We have
\begin{itemize}
\item[(a)] If $\bar{M}$ is proper-biharmonic, then
$\vert\bar{H}\vert^2\in(0,\frac{\bar{m}+3}{\bar{m}}]$.
\item[(b)] If $\vert\bar{H}\vert^2=\frac{\bar{m}+3}{\bar{m}}$, then
$\bar{M}$ is proper-biharmonic if and only if it is pseudo-umbilical
and $\nabla^{\perp}\bar{H}=0$.
\end{itemize}
\end{proposition}
\begin{proof}
Let $\bar{M}$ be a real submanifold of $\cp^n$ of dimension
$\bar{m}$ such that $\jc\bar{H}$ is tangent to $\bar{M}$. Assume
that it has non-zero constant mean curvature, and it is biharmonic.
As $\bar{M}$ is biharmonic we have
$$
\Delta^{\perp}\bar{H}=(\bar{m}+3)\bar{H}- \trace\bar{B}(\cdot,
\bar{A}_{\bar{H}}(\cdot)),
$$
so
$$
\langle\Delta^{\perp}\bar{H},\bar{H}\rangle=(\bar{m}+3)|\bar{H}|^2-
\sum_{i=1}^{\bar{m}}\langle\bar{B}(X_i,
\bar{A}_{\bar{H}}(X_i)),\bar{H}\rangle=(\bar{m}+3)|\bar{H}|^2-|\bar{A}_{\bar{H}}|^2.
$$
Replacing in the Weitzenb\"{o}ck formula (see, for example,
~\cite{JELL})
$$
\frac{1}{2}\Delta
|\bar{H}|^2=\langle\Delta^{\perp}\bar{H},\bar{H}\rangle -
|\nabla^{\perp}\bar{H}|^2
$$
the expression of $\langle\Delta^{\perp}\bar{H},\bar{H}\rangle$, and
using the fact that $|\bar{H}|$ is constant, we obtain
\begin{equation}
\label{eq:firstequation}
(\bar{m}+3)|\bar{H}|^2=|\bar{A}_{\bar{H}}|^2
+|\nabla^{\perp}\bar{H}|^2.
\end{equation}
Let $p$ be an arbitrary point of $\bar{M}$ and let
$\{X_i\}_{i=1}^{\bar{m}}$ be an orthonormal basis of $T_p\bar{M}$
such that $\bar{A}_{\bar{H}}(X_i)=\lambda_iX_i$. We have
$$
\lambda_i=\langle\bar{A}_{\bar{H}}(X_i),X_i\rangle=\langle
\bar{B}(X_i,X_i),\bar{H}\rangle
$$
which implies
$$
\sum_{i=1}^{\bar{m}}\lambda_i=\bar{m}|\bar{H}|^2
$$
or, equivalently,
$$
|\bar{H}|^2=\frac{\sum_{i=1}^{\bar{m}}\lambda_i}{\bar{m}}.
$$
Then the square of the norm of $\bar{A}_{\bar{H}}$ becomes
$$
|\bar{A}_{\bar{H}}|^2=\sum_{i=1}^{\bar{m}}\langle
\bar{A}_{\bar{H}}(X_{i}), \bar{A}_{\bar{H}}(X_i)\rangle
=\sum_{i=1}^{\bar{m}}(\lambda_i)^2.
$$
Replacing in \eqref{eq:firstequation} we get
$$
\frac{\bar{m}+3}{\bar{m}}\sum_i\lambda_i=\sum_i(\lambda_i)^2+
|\nabla^{\perp}\bar{H}|^{2}\geq
\frac{(\sum_i\lambda_i)^2}{\bar{m}}+|\nabla^{\perp}\bar{H}|^2.
$$
Therefore
$$
(\bar{m}+3)|\bar{H}|^2\geq \bar{m}|\bar{H}|^4+
|\nabla^{\perp}\bar{H}|^{2}\geq \bar{m}|\bar{H}|^4,
$$
so
$$
|\bar{H}|^2\in(0,\frac{\bar{m}+3}{\bar{m}}].
$$

(b) If $|\bar{H}|^2=\frac{\bar{m}+3}{\bar{m}}$ and $\bar{M}$ is
biharmonic, the above inequalities become equalities, and therefore
$\lambda_1=\cdots=\lambda_m$ and $\nabla^{\perp}\bar{H}=0$, i.e.
$\bar{M}$ is pseudo-umbilical and $\nabla^{\perp}\bar{H}=0$.

\noindent Conversely, it is clear that if
$|\bar{H}|^2=\frac{\bar{m}+3}{\bar{m}}$ and $\bar{M}$ is
pseudo-umbilical with $\nabla^{\perp}\bar{H}=0$, then $\bar{M}$ is
proper-biharmonic.
\end{proof}

\begin{remark}
We shall see in Proposition~\ref{p1s4} that the upper bound of
$|\bar{H}|^2$ is reached in the case of curves.
\end{remark}

\begin{proposition}
Let $\bar{M}$ be a proper-biharmonic real hypersurface of $\cp^n$ of
constant mean curvature $\vert\bar{H}\vert$. Then its scalar
curvature $s^{\bar{M}}$ is constant and given by
$$
s^{\bar{M}}=4n^2-2n-4+(2n-1)^2\vert\bar{H}\vert^2.
$$
\end{proposition}
\begin{proof}
Let $\bar{M}^{2n-1}$ be a proper-biharmonic real hypersurface of
$\cp^n$ with constant mean curvature, so $|\bar{B}|^{2}=2(n+1)$.

The Gauss equation for the submanifold $\bar{M}$ of $\cp^n$ is
\begin{eqnarray}
\label{eq:Gaussequation}
\langle R^{\bar M}(X,Y)Z,T\rangle&=&\langle R^{\cp^n}(X,Y)Z,T\rangle\\
\nonumber &&-
\langle\bar{B}(Y,T),\bar{B}(X,Z)\rangle+\langle\bar{B}(X,T),\bar{B}(Y,Z)\rangle,
\end{eqnarray}
where $R^{\bar M}$ is the curvature tensor field of $\bar M$.

Let us denote by  $\rho^{\bar M}(X,Y)=\trace\{Z\to R^{\bar
M}(Z,X)Y\}$ the Ricci tensor.

\noindent Computing \eqref{eq:Gaussequation} for $X=T=X_i$, where
$\{X_i\}_{i=1}^{2n-1}$ is a local orthonormal frame field, we have
\begin{eqnarray*}
\langle R^{\bar M}(X_i,Y)Z,X_i\rangle &=&\langle\langle Z,Y\rangle
X_i
-\langle Z,X_i\rangle Y,X_i\rangle\\
&&+\langle\langle\jc Y,Z\rangle\jc X_i,X_i\rangle -
\langle\langle\jc X_i,Z\rangle\jc Y,X_i\rangle\\
&&+2\langle\langle X_i,\jc Y\rangle\jc Z,X_i\rangle\\
&&-\langle\bar{B}(Y,X_i),\bar{B}(X_i,Z)\rangle +
\langle\bar{B}(X_i,X_{i}),\bar{B}(Y,Z)\rangle\\
&=&\langle Z,Y\rangle-\langle Z,X_i\rangle\langle Y,X_i\rangle
\\
&&+\langle\jc Y,Z\rangle\langle\jc X_i,X_i\rangle-\langle\jc
X_i,Z\rangle\langle\jc Y,X_i\rangle
\\
&&+2\langle X_i,\jc Y\rangle\langle\jc Z,X_i\rangle -
\langle\bar{B}(Y,X_i),\bar{B}(Z,X_i)\rangle\\
&& +
\langle\bar{B}(X_i,X_i),\bar{B}(Y,Z)\rangle\\
&=&\langle Z,Y\rangle-\langle Z,X_i\rangle\langle Y,X_i\rangle
+3\langle \jc Z,X_i\rangle \langle\jc Y,X_i
\rangle\\
&&-\langle\bar{A}(Y),X_i\rangle\langle\bar{A}(Z),X_i\rangle
+\langle\bar{B}(X_i,X_i), \bar{B}(Y,Z)\rangle,
\end{eqnarray*}
where $\bar{H}=|\bar{H}|\bar{\eta}$ and
$\bar{A}=\bar{A}_{\bar{\eta}}$. Therefore
\begin{eqnarray*}
\rho^{\bar M}(Y,Z)&=&\sum_{i=1}^{2n-1}\langle R^{\bar M}(X_{i},Y) Z,X_{i}\rangle\\
&=&(2n-1)\langle Z,Y\rangle-\langle Z,Y\rangle + 3\langle
(\jc Z)^{\top}, (\jc Y)^{\top} \rangle\\
&&-\langle\bar{A}(Y),\bar{A}(Z) \rangle
+(2n-1)|\bar{H}|\langle\bar{A}(Y),Z\rangle.
\end{eqnarray*}
Now,
\begin{eqnarray*}
\langle \jc Z,\jc Y\rangle&=&\langle Z,Y\rangle\\
&=& \langle(\jc Z)^{\top}+\langle \jc Z,\bar{\eta}\rangle
\bar{\eta},(\jc Y)^{\top}+\langle\jc Y,\bar{\eta}\rangle
\bar{\eta}\rangle\\
&=&\langle(\jc Z)^{\top},(\jc Y)^{\top}\rangle+\langle \jc
Z,\bar{\eta}\rangle\langle\jc Y,\bar{\eta}\rangle,
\end{eqnarray*}
which implies
$$
\langle(\jc Z)^{\top},(\jc Y)^{\top}\rangle=\langle Z,Y\rangle-
\langle Z,\jc\bar{\eta}\rangle\langle Y,\jc\bar{\eta}\rangle.
$$
Replacing in the above expression of the Ricci tensor, we get
\begin{eqnarray*}
\rho^{\bar M}(Y,Z)&=&2(n-1)\langle Z,Y\rangle+3\{\langle
Y,Z\rangle-
\langle Z,\jc\overline{\eta}\rangle\langle Y,\jc\bar{\eta}\rangle\}\\
&&-\langle\bar{A}(Y),\bar{A}(Z)\rangle+(2n-1)|\bar{H}|\langle\bar{A}(Y),Z
\rangle.
\end{eqnarray*}
Finally, taking the trace, we have
\begin{eqnarray*}
s^{\bar{M}}&=&\sum_{i=1}^{2n-1}\rho^{\bar M}(X_i,X_i)=2(n-1)(2n-1)+
3(2n-1)\\
&&-|\jc\bar{\eta}|^{2}-
|\bar{A}|^2+(2n-1)^2|\bar{H}|^2\\
&=&(2n-2+3)(2n-1)-1 -2(n+1)+(2n-1)^{2}|\bar{H}|^2\\
&=& 4n^{2} -2n-4 +(2n-1)^{2}|\bar{H}|^2.
\end{eqnarray*}
\end{proof}

Another important family of submanifolds  of $\cp^n$ is that
consisting of the submanifolds for which  $\jc\bar{H}$ is normal to
$\bar{M}$. In this case, using an argument similar to the case when
$\jc\bar{H}$ is tangent to $\bar{M}$, we have the following result

\begin{proposition}
Let $\bar{M}$ be a real submanifold of $\cp^n$ of dimension
$\bar{m}$ such that $\jc\bar{H}$ is normal to $\bar{M}$. Then
$\bar{M}$ is biharmonic if and only if
\begin{equation}
\label{eq:caracterizarebiarmonicitate2}
\begin{cases}
\Delta^{\perp}\bar{H}+\trace
\bar{B}(\cdot,\bar{A}_{\bar{H}}(\cdot))-\bar{m}\bar{H}=0\\
4\trace\bar{A}_{\nabla^{\perp}_{(\cdot)}\bar{H}}(\cdot)+\bar{m}\grad
(\vert\bar{H}\vert^2)=0
\end{cases}.
\end{equation}
Moreover, if $\jc\bar{H}$ is normal to $\bar{M}$ and $\bar{M}$ has
parallel mean curvature, then $\bar{M}$ is biharmonic if and only if
$$
\trace \bar{B}(\cdot,\bar{A}_{\bar{H}}(\cdot))=\bar{m}\bar{H}.
$$
\end{proposition}

Also in this case, if the mean curvature is constant we can bound
its value, as it is shown by the following

\begin{proposition}
Let $\bar{M}$ be a real submanifold of $\cp^n$ of dimension
$\bar{m}$ such that $\jc\bar{H}$ is normal to $\bar{M}$. Assume that
it has non-zero constant mean curvature. We have
\begin{itemize}
\item[(a)] If $\bar{M}$ is proper-biharmonic, then
$\vert\bar{H}\vert^2\in(0,1]$.
\item[(b)] If $\vert\bar{H}\vert^2=1$, then
$\bar{M}$ is proper-biharmonic if and only if it is pseudo-umbilical
and $\nabla^{\perp}\bar{H}=0$.
\end{itemize}
\end{proposition}

\begin{remark}
We shall see in Proposition~\ref{p2s4} (a), that the upper bound is
reached in the case of curves.
\end{remark}

\section{The Hopf fibration and the biharmonic
equation}\label{sec:hopfbiharmonic}

Let $\pi:\C^{n+1}\setminus\{0\}\to\cp^n$ be the natural projection.
Then $\pi$ restricted to the sphere $\s^{2n+1}$ of $\C^{n+1}$ gives
rise to the Hopf fibration $\pi:\s^{2n+1}\to\cp^n$ and if $4c=4$
then $\pi:\s^{2n+1}\to\cp^n$ defines a Riemannian submersion. In the
sequel we shall look at $\s^{2n+1}$ as a hypersurface of $\r^{2n+2}$
and we shall denote by $\jr$ the complex structure of $\r^{2n+2}$.

Let $\bar{M}$ be a real submanifold of $\cp^n$ of dimension
$\bar{m}$ and denote by $M:=\pi^{-1}(\bar{M})$ the Hopf-tube over
$\bar{M}$. If we denote by $\ic:\bar{M}\to\cp^n$ and
$\is:M\to\s^{2n+1}$ the respective inclusions we have the following
diagram
$$
\begin{CD}
M @>\is>> \s^{2n+1}\\
@V VV @V  VV\pi\\
\bar{M} @>\ic>> \cp^n.
\end{CD}
$$

We shall now find the relation between the bitension field of the
inclusion $\ic$ and the bitension field of the inclusion $\is$. For
this, let $\{\bar{X}_k\}_{k=1}^{\bar{m}}$ be a local orthonormal
frame field tangent to $\bar{M}$, $1\leq\bar{m} \leq 2n-1$, and let
$\{\bar{\eta}_{\alpha}\}_{\alpha=\bar{m}+1}^{2n}$ be a local
orthonormal frame field normal to $\bar{M}$. Let us denote by
$X_k:=\bar{X}_k^{H}$ and $\eta_{\alpha}:=\bar{\eta}_{\alpha}^H$ the
horizontal lifts with respect to the Hopf map and by $\xi$ the Hopf
vector field on $\s^{2n+1}$ which is tangent to the fibres of the
Hopf fibration, i.e. $\xi(p)=-\jr p$, for any $p\in\s^{2n+1}$. Then
$\{\xi,X_k\}$ is a local orthonormal frame field tangent to $M$ and
$\{\eta_{\alpha}\}$ is a local orthonormal frame field normal to
$M$.

\begin{lemma}\label{lem:nablas}
Let $X=\bar{X}^{H}\in C(TM)$, where $\bar{X}\in C(T\bar{M})$, and
$V=\bar{V}^{H}\in C(\is^{-1}(T\s^{2n+1}))$, where $\bar{V}\in
C((\ic)^{-1}(T\cp^{n}))$. Then
$$
\nabla_X^{\is} V=(\nabla_{\bar{X}}^{\ic}\bar{V})^H+\langle V,\jr
X\rangle \xi = (\nabla_{\bar{X}}^{\ic}\bar{V})^H+(\langle
\bar{V},\jc \bar{X}\rangle\circ\pi) \xi,
$$
where $\nabla^{\is}$ and $\nabla^{\ic}$ denote the pull-back
connections on $\is^{-1}(T\s^{2n+1})$ and $(\ic)^{-1}(T\cp^{n})$,
respectively.
\end{lemma}
\begin{proof}
Decomposing $\nabla_X^{\is} V$ in its horizontal and vertical
components we have
$$
\nabla_X^{\is}
V=\nabla_{\bar{X}^H}^{\is}\bar{V}^H=(\nabla_{\bar{X}}^{\ic}\bar{V})^H+\langle
\nabla_X^{\is}V,\xi\rangle \xi.
$$
Now,
\begin{eqnarray*}
\langle \nabla_X^{\is}V,\xi\rangle &=& -\langle V,\nabla_X^{\is}\xi\rangle=-\langle
V,\nablar_X\xi+\langle X,\xi \rangle p\rangle\\
&=& \langle V,\nablar_X \jr p \rangle =\langle V, \jr X
\rangle=\langle \bar{V}, \jc \bar{X} \rangle\circ\pi,
\end{eqnarray*}
where $\nablar$ is the Levi-Civita connection on the Euclidean
space $\mathbb{E}^{2n+2}$.
\end{proof}

\begin{lemma}\label{lem:laplacian}
If $V=\bar{V}^{H}\in C(\is^{-1}(T\s^{2n+1}))$, $\bar{V}\in
C((\ic)^{-1}(T\cp^{n}))$, then
$$
\Delta^{\is} V=(\Delta^{\ic}\bar{V})^H+2 \di((\jr
V)^{\top})\xi+\langle V,\jr \tau(\is)\rangle\xi +V-\jr(\jr
V)^{\top},
$$
where $\Delta^{\is}$ and $\Delta^{\ic}$ are the rough Laplacians
acting on sections of $\is^{-1}(T\s^{2n+1})$ and
$(\ic)^{-1}(T\cp^{n})$, respectively, whilst $(V)^{\top}$ denotes
the component of $V$ tangent to $M$.
\end{lemma}
\begin{proof} The Laplacian $\Delta^{\is}$ is given by
$$
-\Delta^{\is} V = \sum_{i=1}^{\bar{m}}\{ \nablas_{X_i}\nablas_{X_i}
V- \nablas_{\nabla^M_{X_i}{X_i}}V\}+ \nablas_{\xi}\nablas_{\xi} V-
\nablas_{\nabla^M_{\xi}{\xi}}V.
$$
We compute each term separately. From Lemma~\ref{lem:nablas} we have
\begin{eqnarray*}
\nablas_{X_i}\nablas_{X_i} V&=&(\nablac_{\bx_i}\nablac_{\bx_i}
\bv)^H+\langle \nablas_{X_i} V,\jr X_i \rangle\xi
+\nablas_{X_i}(\langle V,\jr X_i\rangle\xi)\\
&=& (\nablac_{\bx_i}\nablac_{\bx_i} \bv)^H+2\langle \nablas_{X_i} V,\jr X_i \rangle\xi\\
&&+\langle V, \nablas_{X_i}\jr X_i\rangle\xi + \langle \jr V, X_i
\rangle \jr{X_i}.
\end{eqnarray*}
Using
$$
\nablas_{X_i}\jr X_i=\jr \nablas_{X_i} X_i+\xi
$$
we get
\begin{eqnarray}\label{eq:delta1}
\nablas_{X_i}\nablas_{X_i} V&=& (\nablac_{\bx_i}\nablac_{\bx_i} \bv)^H+2\langle \nablas_{X_i} V,\jr X_i \rangle\xi\\
&&+\langle V, \jr \nablas_{X_i} X_i\rangle\xi + \jr(\langle \jr V,
X_i \rangle {X_i})\nonumber.
\end{eqnarray}
Next
\begin{equation}\label{eq:delta2}
\nablas_{\nabla^M_{X_i}{X_i}}V=
(\nablac_{\nabla^{\bar{M}}_{\bx_i}{\bx_i}}\bv)^H+\langle V,\jr
\nabla^M_{X_i}{X_i} \rangle\xi.
\end{equation}
Summing \eqref{eq:delta1} and \eqref{eq:delta2} up we find
\begin{eqnarray*}
-\Delta^{\is} V&=&-(\Delta^{\ic}\bar{V})^H + 2
\sum_{i=1}^{\bar{m}} \langle \nablas_{X_i} V,\jr X_i \rangle\xi
+\langle V,\jr \sum_{i=1}^{\bar{m}} (\nablas_{X_i} X_i-\nabla^M_{X_i}{X_i} )\rangle\xi  \\
&&+\sum_{i=1}^{\bar{m}} \jr(\langle \jr V, X_i \rangle {X_i})+ \nablas_{\xi}\nablas_{\xi} V\\
&=&-(\Delta^{\ic}\bar{V})^H + 2 \sum_{i=1}^{\bar{m}} \langle
\nablas_{X_i} V,\jr X_i \rangle\xi
+\langle V,\jr \tau(\is)\rangle\xi  \\
&&+ \jr(\jr V)^{\top}+ \nablas_{\xi}\nablas_{\xi} V.
\end{eqnarray*}
We now compute the extra terms in the above equation.
\begin{eqnarray}
\sum_{i=1}^{\bar{m}}\langle\nablas_{X_i}V,\jr X_i\rangle&=&
\sum_{i=1}^{\bar{m}}\{-X_i\langle\jr V,X_i\rangle+\langle\jr V,\nablas_{X_i}X_i\rangle\}\\
&=&\langle\jr V,\tau(\is)\rangle-\sum_{i=1}^{\bar{m}}\{X_i\langle
\jr V,X_i\rangle-
\langle\jr V,\nabla^M_{X_i}X_i\rangle\}\nonumber\\
&=&\langle\jr V,\tau(\is)\rangle-\di((\jr V)^{\top}).\nonumber
\end{eqnarray}
Finally
\begin{eqnarray*}
\nablas_{\xi}V&=&H(\nablas_{\xi}V)+\langle\nablas_{\xi}V,\xi\rangle\xi=H(\nablas_{\xi}V)\\
&=&H(\nablas_{V}\xi)=H(\nablar_{V}\xi+\langle V,\xi\rangle p)=
H(-\jr V)=-\jr V
\end{eqnarray*}
which gives
$$
\nablas_{\xi}\nablas_{\xi} V=-V
$$
\end{proof}

Before giving the relation between the bitension fields we need to
compute the trace of the curvature operators. One gets immediately
\begin{equation}\label{ea:curvature-s}
-\trace R^{\s^{2n+1}}(d\is,\tau(\is))d\is=(\bar{m}+1)\tau(\is)
\end{equation}
and
\begin{equation}\label{ea:curvature-c}
-\trace
R^{\cp^{n}}(d\ic,\tau(\ic))d\ic=\bar{m}\tau(\ic)-3\jc(\jc\tau(\ic))^{\top}.
\end{equation}

We are now ready to state the main theorem of this section

\begin{theorem}\label{teo:reltautau}
Let $\bar{M}$ be a real submanifold of $\cp^n$ of dimension
$\bar{m}$ and denote by $M:=\pi^{-1}(\bar{M})$ the corresponding
Hopf-tube. If we denote by $\ic:\bar{M}\to\cp^n$ and
$\is:M\to\s^{2n+1}$ the respective inclusions we have that
\begin{equation}\label{eq:bitension-link}
(\tau_2(\ic))^H=\tau_2(\is)-4\jr(\jr\tau(\is))^{\top}+2 \di((\jr
\tau(\is))^{\top})\xi.
\end{equation}
\end{theorem}
\begin{proof}
From \eqref{eq-tau2} and \eqref{ea:curvature-s} we have
$$
\tau_2(\is)=-\Delta^{\is}\tau(\is)+(\bar{m}+1)\tau(\is).
$$
Next, since $\tau(\is)=(\tau(\ic))^H$, using
Lemma~\ref{lem:laplacian} and \eqref{ea:curvature-c} we find the
assertion of the theorem.
\end{proof}

\begin{remark}\label{rem:div0}
\begin{itemize}
\item[(i)] Using the horizontal lift, it is straightforward to
check that \eqref{eq:bitension-link}
can be written as
$$
(\tau_2(\ic))^H=\tau_2(\is)-4(\jc(\jc\tau(\ic))^{\top})^H+2
(\di_{\bar{M}}((\jc \tau(\ic))^{\top})\circ\pi) \xi.
$$
\item[(ii)] If $\jr\tau(\is)$ is normal to $M$, then
$\tau_2(\ic)=0$ if and only if $\tau_2(\is)=0$. \item[(iii)] If
$\jr\tau(\is)$ is tangent to $M$, then $\tau_2(\ic)=0$ and
$\di_{\bar{M}}((\jc \tau(\ic))^{\top})=0$ if and only if
$\tau_2(\is)+4\tau(\is)=0$. \item[(iv)] Assume that, locally,
$M=\pi^{-1}(\bar{M})=\s^1\times\tilde{M}$, where $\tilde{M}$ is an
integral submanifold of $\s^{2n+1}$, i.e. $\langle
\tilde{X}_{\tilde p},\xi(\tilde{p})\rangle=0$, for any vector
$\tilde{X}_{\tilde p}$ tangent to $\tilde{M}$. Denote by
$\tilde{\is}:\tilde{M}\to \s^{2n+1}$ the canonical inclusion, and
by $\{\phi_t\}$ the flow of $\xi$. We know that
$\tau_2(\is)_{(t,\tilde{p})}=(d\phi_t)_{\tilde{p}}(\tau_2(\tilde{\is}))$,
see \cite{DFCO}, and we can check that, at $\tilde{p}$,
$$
(\tau_2(\ic))^H=\tau_2(\tilde\is)-4\jr(\jr\tau(\tilde{\is}))^{\top}+2
\di_{\tilde{M}}((\jr \tau(\tilde{\is}))^{\top})\xi.
$$
\end{itemize}
\end{remark}

To state the next results we recall that a
 smooth map $\varphi:(M,g)\to (N,h)$ is called {\it
$\lambda$-biharmonic} if it is a critical point of the {\it
$\lambda$-bienergy}
$$
E_2(\varphi)+\lambda E(\varphi),
$$
where $\lambda$ is a real constant. The critical points of the
$\lambda$-bienergy satisfy the equation
$$
\tau_2(\varphi)-\lambda\tau(\varphi)=0.
$$

\begin{proposition}
Let $\bar{M}$ be a real hypersurface of $\cp^n$ of constant mean
curvature and denote by $M=\pi^{-1}(\bar{M})$ the Hopf-tube over
$\bar{M}$. Then $\tau_2(\ic)=0$ if and only if
$\tau_2(\is)+4\tau(\is)=0$, i.e. $\is$ is $(-4)$-biharmonic.
\end{proposition}
\begin{proof}
We have $(\jc\tau(\ic))^{\top}=\jc\tau(\ic)$ and it remains to
prove that $\di_{\bar{M}}(\jc\tau(\ic))=0$. Let $\bar{\eta}$ be a
local unit section in the normal bundle of $\bar{M}$ in $\cp^n$
and consider $\{\bar{X}_1,\jc\bar{X}_1, \ldots,
\bar{X}_{n-1},\jc\bar{X}_{n-1}, \jc\bar{\eta}\}$ a local
orthonormal frame field tangent to $\bar{M}$. Since $\bar{M}$ is a
hypersurface of constant mean curvature, it is enough to prove
that $\di_{\bar{M}}(\jc\bar{\eta})=0$. But, denoting by
$\bar{A}_{\bar{\eta}}$ the shape operator of $\bar{M}$,
$$
\langle\nabla_{\bar{X}_a}^{\bar{M}}\jc\bar{\eta},\bar{X}_a\rangle=\langle
\bar{A}_{\bar{\eta}}(\bar{X}_a),\jc\bar{X}_a\rangle, \quad
\langle\nabla_{\jc\bar{X}_b}^{\bar{M}}\jc\bar{\eta},\jc\bar{X}_b\rangle=-\langle
\bar{A}_{\bar{\eta}}(\bar{X}_b),\jc\bar{X}_b\rangle,
$$
for any $1\leq a,b\leq n-1$, and
$$\langle\nabla_{\jc\bar{\eta}}^{\bar{M}}\jc\bar{\eta},\jc\bar{\eta}\rangle=0,
$$
so we conclude.
\end{proof}

\begin{proposition}
Let $\bar{M}$ be a Lagrangian submanifold of $\cp^n$ with parallel
mean curvature vector field and denote by $M=\pi^{-1}(\bar{M})$ the
Hopf-tube over $\bar{M}$. Then $\ic$ is biharmonic if and only if
$\is$ is $(-4)$-biharmonic.
\end{proposition}

\begin{proof}
Since $\bar{M}$ is a Lagrangian submanifold, $\dim
\bar{M}=\bar{m}=n$ and $\jc(T\bar{M})=N\bar{M}$ (therefore
$\jc(N\bar{M})=T\bar{M}$). We have that $\jc\tau(\ic)\in
C(T\bar{M})$ and we shall prove that
$\nabla^{\bar{M}}\jc\tau(\ic)=0$ which implies $\di_{\bar{M}}(\jc
\tau(\ic))=0$. Indeed, for any $\bar{X}$ and $\bar{Y}$ tangent to
$\bar{M}$ we have
\begin{eqnarray*}
\langle\nabla^{\bar{M}}_{\bar{X}}\jc\tau(\ic),\bar{Y}\rangle&=&
\langle\nabla^{\ic}_{\bar{X}}\jc\tau(\ic),\bar{Y}\rangle=
\langle\jc\nabla^{\ic}_{\bar{X}}\tau(\ic),\bar{Y}\rangle=
\langle-\jc \bar{A}_{\tau(\ic)}(\bar{X}),\bar{Y}\rangle \\
&=&0.
\end{eqnarray*}
\end{proof}

We end this section with

\begin{proposition}
Let $\bar{M}$ be a real submanifold of $\cp^n$ such that $\jc
\tau(\ic)$ is normal to $\bar{M}$ and denote by
$M=\pi^{-1}(\bar{M})$ the Hopf-tube over $\bar{M}$. Then $\ic$ is
biharmonic if and only if $\is$ is biharmonic.
\end{proposition}

\section{Biharmonic submanifolds of Clifford type}

For a fixed $n>1$, consider the spheres
$\s^{2p+1}(a)\subset\r^{2p+2}=\C^{p+1}$ and
$\s^{2q+1}(b)\subset\r^{2q+2}=\C^{q+1}$, with $a^2+b^2=1$ and
$p+q=n-1$. Denote by
$T^{p,q}_{a,b}=\s^{2p+1}(a)\times\s^{2q+1}(b)\subset \s^{2n+1}$ the
Clifford torus. Let now $M_1$ be a minimal submanifold of
$\s^{2p+1}(a)$ of dimension $m_1$ and $M_2$ a minimal submanifold of
$\s^{2q+1}(b)$ of dimension $m_2$. The submanifold $M_1\times M_2$
is clearly minimal in $T^{p,q}_{a,b}$ and, according to
\cite{RCSMCO2}, is proper biharmonic in $\s^{2n+1}$ if and only if
$a=b=\sqrt{2}/2$ and $m_1\neq m_2$. If $M_1\times M_2$ is invariant
under the action of the one-parameter group of isometries generated
by the Hopf vector field $\xi$ on $\s^{2n+1}$, then it projects onto
a submanifold of $\cp^n$ and we could ask for which values of
$a,b,m_1,m_2$ is it a proper-biharmonic submanifold.

We start with the following

\begin{lemma}\label{lem:tau-tau2-torus}
Let denote by  $\is_1:M^{m_1}_1\times M^{m_2}_2\to T^{p,q}_{a,b}$
the inclusion of $M_1\times M_2$ in the Clifford torus and by $\is:
T^{p,q}_{a,b}\to \s^{2n+1}$ the inclusion of the Clifford torus in
the sphere. Then
\begin{equation}
\begin{cases}
\tau(\is\circ\is_1)=(\dfrac{a}{b}m_2-\dfrac{b}{a}m_1)\eta=c\eta\\
\tau_2(\is\circ\is_1)=c(m_1+m_2-\dfrac{b^2}{a^2}m_1-\dfrac{a^2}{b^2}m_2)\eta
\end{cases},
\end{equation}
where $\eta$ is the unit normal section in the normal bundle of
$T^{p,q}_{a,b}$ in $\s^{2n+1}$ given by $\eta(x,y)=(\frac{b}{a} x,
-\frac{a}{b}y)$, $x\in\s^{2p+1}(a), y\in\s^{2q+1}(b)$.
\end{lemma}
\begin{proof}
Let $p=(x,y)\in T^{p,q}_{a,b}$, $x\in\r^{2p+2}$, $y\in\r^{2q+2}$,
$|x|=a$, $|y|=b$. Then $\eta(x,y)=(\frac{b}{a} x, -\frac{a}{b}y)$
defines a unit normal section in the normal bundle of
$T^{p,q}_{a,b}$ in $\s^{2n+1}$. We identify $X=(X,0)\in
T_pT^{p,q}_{a,b}$, $Y=(0,Y)\in T_pT^{p,q}_{a,b}$, and a
straightforward computation gives
$$
\nabla^{\is}_X \eta=-A^{\is}(X)=\frac{b}{a}X, \quad \nabla^{\is}_Y
\eta=-A^{\is}(Y)=-\frac{a}{b}Y.
$$
Let $\{X_k=(X_k,0)\}$ be a local orthonormal frame field tangent to
$\s^{2p+1}(a)$ and $\{Y_l=(0,Y_l)\}$ a local orthonormal frame field
tangent to $\s^{2q+1}(b)$. Then, applying the composition law for
the tension field and using that $\is_1$ is harmonic, we have
\begin{eqnarray*}
\tau(\is\circ\is_1)&=&d\is(\tau(\is_1))+\trace \nabla d\is(d\is_1,d\is_1)\\
&=&\sum_{k=1}^{m_1}\langle A^{\is}(X_k),X_k\rangle\eta
+\sum_{l=1}^{m_2}\langle
A^{\is}(Y_l),Y_l\rangle\eta=(\dfrac{a}{b}m_2-\dfrac{b}{a}m_1)\eta=c\,\eta.
\end{eqnarray*}
To compute $\tau_2(\is\circ\is_1)$, let us choose around $p=(x,y)\in
M_1\times M_2$ a frame field $\{(X_k,Y_l)\}$ such that
$\{X_k\}_{k=1}^{m_1}$ is a geodesic frame field around $x$ and
$\{Y_l\}_{l=1}^{m_2}$ is a geodesic frame field around $y$. Then at
$p$
\begin{eqnarray}\label{eq:laplacian-cliff}
-\Delta^{\is\circ\is_1}\eta &=&\sum_{k=1}^{m_1}
\nabla^{\is\circ\is_1}_{X_k}\nabla^{\is\circ\is_1}_{X_k} \eta+
\sum_{l=1}^{m_2} \nabla^{\is\circ\is_1}_{Y_l}\nabla^{\is\circ\is_1}_{Y_l}\eta\nonumber \\
&=& \frac{b}{a} \sum_{k=1}^{m_1}\nabla^{\is\circ\is_1}_{X_k}X_k-
\frac{a}{b}
\sum_{l=1}^{m_2}\nabla^{\is\circ\is_1}_{Y_l}Y_l \nonumber\\
&=&\frac{b}{a}\sum_{k=1}^{m_1}(B^{\is}(X_k,X_k)+
\nabla^{T^{p,q}_{a,b}}_{X_k}X_k)
-\frac{a}{b}\sum_{l=1}^{m_2}(B^{\is}(Y_l,Y_l)+\nabla^{T^{p,q}_{a,b}}_{Y_l}Y_l) \\
&=&\frac{b}{a}\sum_{k=1}^{m_1}B^{\is}(X_k,X_k)
-\frac{a}{b}\sum_{l=1}^{m_2}B^{\is}(Y_l,Y_l)\nonumber\\
&=& \frac{b^2}{a^2}m_1-\frac{a^2}{b^2}m_2.\nonumber
\end{eqnarray}
Finally, using the standard formula for the curvature of
$\s^{2n+1}$, we get
$$
-\trace
R^{\s^{2n+1}}(d(\is\circ\is_1),\tau(\is\circ\is_1))d(\is\circ\is_1)=(m_1+m_2)\tau(\is\circ\is_1)
=(m_1+m_2)c\eta,
$$
that summed up with \eqref{eq:laplacian-cliff} gives the lemma.
\end{proof}

\begin{theorem}
\label{eq:Cliffordtype} Let $\pi:\s^{2n+1}\to\cp^n$ be the Hopf map.
Let $M=M^{m_1}_1\times M^{m_2}_2$ be the product of two minimal
submanifolds of $\s^{2p+1}(a)$ and $\s^{2q+1}(b)$, respectively.
Assume that $M$ is invariant under the action of the one-parameter
group of isometries generated by the Hopf vector field $\xi$ on
$\s^{2n+1}$. Then $\pi(M)$ is a proper-biharmonic submanifold of
$\cp^n$ if and only if $M$ is $(-4)$-biharmonic, that is
\begin{equation}
\begin{cases}
a^2+b^2=1\\
\dfrac{a}{b}m_2-\dfrac{b}{a}m_1\neq 0 \\
\dfrac{b^2}{a^2}m_1+\dfrac{a^2}{b^2}m_2=4+m_1+m_2
\end{cases},
\end{equation}
where $m_1$ and $m_2$ are the dimensions of $M_1$ and $M_2$,
respectively.
\end{theorem}

\begin{proof}
The Hopf vector field $\xi$ is a Killing vector field on $\s^{2n+1}$
that, at a point $p=(x,y)$, is given by
$$
\xi=-(-x^2,x^1,\ldots,-x^{2p+2},x^{2p+1},-y^2,y^1,\ldots,-y^{2q+2},y^{2q+1})=(\xi_1,\xi_2).
$$
Since  $M_1\times M_2$ is invariant under the action of the
one-parameter group of isometries generated by $\xi$, it remains
Killing when restricted to  $M_1\times M_2$. As
$$
\jr\eta=(-\frac{b}{a}\xi_1, \frac{a}{b}\xi_2),
$$
it follows that $\jr\eta$ is a Killing vector field on $M_1\times
M_2$.

Since $\di(\jr \tau(\is\circ\is_1))=\di (c\jr\eta)=0$, using
Remark~\ref{rem:div0} (iii), it results that
 $\pi(M_1\times M_2)$ is a biharmonic submanifold of $\cp^n$ if and only if
$$
\tau_2(\is\circ\is_1)+4\tau(\is\circ\is_1)=0.
$$
Finally, using Lemma~\ref{lem:tau-tau2-torus}, we get
$$
\tau_2(\is\circ\is_1)+4\tau(\is\circ\is_1)=c(4+m_1+m_2
-\dfrac{b^2}{a^2}m_1-\dfrac{a^2}{b^2}m_2)\eta.
$$
\end{proof}

\begin{remark}
If $M_1=\s^{2p+1}(a)$ and $M_2=\s^{2q+1}(b)$, we recover the result
in ~\cite{TIJIHU} concerning the proper-biharmonic homogeneous real
hypersurfaces of type $A$ in $\cp^n$.
\end{remark}

\begin{example}
\label{example}
Let $e_1$ and $e_3$ be two constant unit vectors in
$\mathbb{E}^{2n+2}$, with $e_3$ orthogonal to $e_1$ and $\jr e_1$.
We consider the circles $\s^1(a)$ and $\s^1(b)$ lying in the
$2$-planes spanned by $\{e_1,\jr e_1\}$ and $\{e_3,\jr e_3\}$,
respectively. Then $M=\s^1(a)\times\s^1(b)$ is invariant under the
flow-action of $\xi$, and $\pi(M)$ is a proper-biharmonic curve of
$\cp^n$ if and only if $a=\frac{\sqrt{2\pm\sqrt{2}}}{2}$.
\end{example}

\begin{example}
For $p=0$ and $q=n-1$, we get that
$\pi(\s^1(a)\times\s^{2n-1}(b))$ is proper-biharmonic in $\cp^n$ if
and only if $a^2=\frac{n+3\pm\sqrt{n^2+2n+5}}{4(n+1)}$. In
particular, $\pi(\s^1(a)\times\s^3(b))$ is a proper-biharmonic real
hypersurface in $\cp^2$ if and only if
$a^2=\frac{5\pm\sqrt{13}}{12}$.
\end{example}

\begin{example}
If $p=q$ then $M=T^{p,p}_{a,b}$  is never a
proper-biharmonic hypersurface of $\s^{2n+1}$, and it is easy to
check that $\pi(M)$ is a proper-biharmonic hypersurface of $\cp^n$
if and only if $a^2=\frac{2p+2-\sqrt{2(p+1)}}{4(p+1)}$.
\end{example}

\begin{example}
Let
$M=\s^{2p+1}(a)\times\s^p\Big(\frac{b}{\sqrt{2}}\Big)\times
\s^p\Big(\frac{b}{\sqrt{2}}\Big)$, $p$ odd. Then $M$ is minimal in
$T^{p,p}_{a,b}$, and is proper-biharmonic in $\s^{2n+1}$ if and
only if $a=b=\frac{1}{\sqrt{2}}$. By a straightforward computation
we can check that $\pi(M)$ is proper-biharmonic in $\cp^n$ if and
only if $a^2=\frac{8p+7\pm\sqrt{32p+25}}{16p+12}$.
\end{example}

\subsection{Sphere bundle of all vectors tangent to $\s^{2p+1}(a)$}
We have seen that if $M$ is a product submanifold in
$T^{p,q}_{a,b}$ then its projection $\pi(M)$ can be
proper-biharmonic in $\cp^n$. But when $M$ is not a product, the
situation can be more complicated as it is illustrated by the
following example.

We consider the sphere of radius $a$
$$
\s^{2p+1}(a)=\{x\in\r^{2p+2}:(x^1)^2+\cdots+(x^{2p+2})^2=a^2\}
$$
and its sphere bundle of all vectors tangent to $\s^{2p+1}(a)$ and
of norm $b$, that is
$$
M=T^b\s^{2p+1}(a)=\{(x,y)\in\r^{4p+4}:x,y\in\r^{2n+2},|x|=a,|y|=b,\langle
x,y\rangle=0\}.
$$
It is easy to check that $M$ is invariant under the flow-action of
the characteristic vector field $\xi$, which means
$e^{-\mathrm{i}t}p\in M$, $\forall p\in M$ and $\forall t\in\r$.
Let $(x_0,y_0)\in M$. Then
$$
\begin{array}{lll}
T_{(x_0,y_0)}M=\{Z_0=(X_0,Y_0)\in\r^{4p+4}&:&\langle
x_0,X_0\rangle=0, \ \langle y_0,Y_0\rangle=0,\\&&\langle
X_0,y_0\rangle+\langle x_0,Y_0\rangle=0\}.
\end{array}
$$
In order to find a basis in $T_{(x_0,y_0)}M$, we consider
$\{y_0,y_1,\ldots,y_{2p+1}\}$ an orthogonal basis in
$T_{x_0}\s^{2p+1}(a)$, each vector being of norm $b$. We think $M$
as a hypersurface of the tangent bundle $T\s^{2p+1}(a)$, and we
consider on $T\s^{2p+1}(a)$ and $M$ the induced metrics from the
canonical metric on $\r^{4p+1}$
$$
M\hookrightarrow T\s^{2p+1}(a)\hookrightarrow\r^{4p+4}.
$$
The above inclusions are the canonical ones.

The vertical lifts of the tangent vectors $y_2,y_3,\ldots,y_{2p+1}$,
in $(x_0,y_0)$, are
$$
y_2^V=(0,y_2), \ y_3^V=(0,y_3),\ldots, \ y_{2p+1}^V=(0,y_{2p+1}),
$$
and the horizontal lifts of $y_0,y_2,y_3,\ldots,y_{2p+1}$, in
$(x_0,y_0)$, are
$$
y_0^H=(y_0,-\frac{b^2}{a^2}x_0), \ y_2^H=(y_2,0), \ y_3^H=(y_3,0),
\ldots, \ y_{2p+1}^H=(y_{2p+1},0).
$$

The vectors $\{y_0^H,y_2^H,\ldots,y_{2p+1}^H,y_2^V,y_3^V,\ldots,
y_{2p+1}^V\}$ form an orthogonal basis in $T_{(x_0,y_0)}M$ and
$$
|y_2^V|=|y_3^V|=\cdots=|y_{2p+1}^V|=b, \
|y_2^H|=|y_3^H|=\cdots=|y_{2p+1}^H|=b, \ |y_0^H|=\frac{b}{a}.
$$

The vector $C(x_0,y_0)=y_0^V=(0,y_0)$ is tangent to
$T\s^{2p+1}(a)$ in $(x_0,y_0)$ and orthogonal to $M$.

From now on we shall consider $a^2+b^2=1$ and the inclusions
$$
M\hookrightarrow\s^{2p+1}(a)\times\s^{2p+1}(b)\hookrightarrow\s^{4p+3}\hookrightarrow\r^{4p+4}.
$$
We define $\eta_1(x_0,y_0)=(y_0,x_0)$ and
$\eta_2(x_0,y_0)=(x_0,-\frac{a^2}{b^2}y_0)$. We have that $\eta_1$
and $\eta_2$ are normal to $M$, and
$$
\eta_1(x_0,y_0)\in
T_{(x_0,y_0)}(\s^{2p+1}(a)\times\s^{2p+1}(b)),\quad
|\eta_1(x_0,y_0)|=1
$$
$$
\eta_2(x_0,y_0)\in T_{(x_0,y_0)}\s^{4p+3},\ \eta_2(x_0,y_0)\perp
T_{(x_0,y_0)}(\s^{2p+1}(a)\times\s^{2p+1}(b)),\
|\eta_2(x_0,y_0)|=\frac{a}{b}.
$$
We denote by $B_{(x_0,y_0)}$ the second fundamental form of $M$ in
$\s^{4p+3}$, in the point $(x_0,y_0)$. By a straightforward
computation we obtain
\begin{equation}\label{eq:secondfundamentalform}
B_{(x_0,y_0)}(Z_0,Z_0)=-2\langle
X_0,Y_0\rangle\eta_1-\frac{b^2}{a^2}(|X_0|^2-\frac{a^2}{b^2}|Y_0|^2)\eta_2,
\end{equation}
where $Z_0=(X_0,Y_0)\in T_{(x_0,y_0)}M$. From
\eqref{eq:secondfundamentalform} we get
$$
H(x_0,y_0)=\frac{2p}{4p+1}\frac{a^2-b^2}{a^2}\eta_2=c\eta_2.
$$
Therefore $M$ is minimal in $\s^{4p+3}$ if and only if
$a=b=\frac{1}{\sqrt{2}}$.

It is not difficult to check that
\begin{equation}\label{eq:connection}
\begin{cases}
\nabla^{\s^{4p+3}}_{y_0^H}\eta_2=\eta_1,\
\nabla^{\s^{4p+3}}_{y_2^H}\eta_2=y_2^H,\
\nabla^{\s^{4p+3}}_{y_3^H}\eta_2=y_3^H,\ldots,\
\nabla^{\s^{4p+3}}_{y_{2p+1}^H}\eta_2=y_{2p+1}^H\\ \\
\nabla^{\s^{4p+3}}_{y_2^V}\eta_2=-\frac{a^2}{b^2}y_2^V,\
\nabla^{\s^{4p+3}}_{y_3^V}\eta_2=-\frac{a^2}{b^2}y_3^V,\ldots,\
\nabla^{\s^{4p+3}}_{y_{2p+1}^V}\eta_2=-\frac{a^2}{b^2}y_{2p+1}^V\\
\\
\nabla^{\s^{4p+3}}_{y_0^H}\eta_1=-\frac{b^2}{a^2}\eta_2,\
\nabla^{\s^{4p+3}}_{y_2^H}\eta_1=y_2^V,\
\nabla^{\s^{4p+3}}_{y_3^H}\eta_1=y_3^V,\ldots,\
\nabla^{\s^{4p+3}}_{y_{2p+1}^H}\eta_1=y_{2p+1}^V\\ \\
\nabla^{\s^{4p+3}}_{y_2^V}\eta_1=y_2^H,\
\nabla^{\s^{4p+3}}_{y_3^V}\eta_1=y_3^H,\ldots,\
\nabla^{\s^{4p+3}}_{y_{2p+1}^V}\eta_1=y_{2p+1}^H
\end{cases}.
\end{equation}
From \eqref{eq:connection} we obtain that
\begin{equation}\label{eq:A}
\trace
A_{\nabla^{\perp}_{(\cdot)}\eta_2}(\cdot)=0\quad\textnormal{and}\quad\trace
B(\cdot,A_{\eta_2}(\cdot))=2p(\frac{a^2}{b^2}+\frac{b^2}{a^2})\eta_2.
\end{equation}
Denoting $W(x_0,y_0)=y_0^H$, we get
\begin{equation}\label{eq:Delta}
-\Delta^{\perp}\eta_2=\frac{a^2}{b^2}(\nabla^{\perp}_W\nabla^{\perp}_W\eta_2-\nabla^{\perp}_{\nabla^M_W
W}\eta_2)=-\eta_2.
\end{equation}
Before concluding we give the following Lemma which follows by
direct computation.

\begin{lemma}\label{lemmaexample} Let $N^n$ be a hypersurface of a Riemmanian manifold
$(P^{n+1},\langle,\rangle)$, and $X\in C(TP)$ a Killing vector
field. We denote $X^\top=(X_{/N})^\top\in C(TN)$. Then $\di
X^\top=n\langle H,X\rangle$, where $H$ is the mean curvature
vector field of $N$. In particular, if $N$ is minimal then $\di
X^\top=0$.
\end{lemma}

Now we can state

\begin{proposition} Let $M=T^b\s^{2p+1}(a)$ be the sphere bundle of all
vectors of norm $b$ tangent to $\s^{2p+1}(a)$. Assume that
$a^2+b^2=1$ and $p\geq 1$. Then we have
\begin{itemize}
\item[(a)] $M$ is never proper-biharmonic in $\s^{4p+3}$.
\item[(b)] $M$ is $(-4)$-biharmonic in $\s^{4p+3}$ if and only if
$a^2=\frac{2p+1\pm\sqrt{2p+1}}{4p+2}$.
\item[(c)] $M$ is minimal in
$T^{p,p}_{a,b}=\s^{2p+1}(a)\times\s^{2p+1}(b)$.
\item[(d)] $\pi(M)$ is never proper-biharmonic in $\cp^n$.
\end{itemize}
\end{proposition}

\begin{proof} As the mean curvature vector field of $M$ in
$\s^{4p+3}$ is $H=c\eta_2$, where
$c=\frac{2p}{4p+1}\frac{a^2-b^2}{a^2}$, then $M$ is biharmonic if
and only if
\begin{equation}
\begin{cases}
-\Delta^{\perp}\eta_2-\trace
B(\cdot,A_{\eta_2}(\cdot))+(4p+1)\eta_2=0\\
2\trace A_{\nabla^{\perp}_{(\cdot)}\eta_2}(\cdot)
+\frac{4p+1}{2}\grad(c|\eta_2|^2)=0
\end{cases}.
\end{equation}
From \eqref{eq:A} and \eqref{eq:Delta} we get that $M$ is
biharmonic if and only if
$$
-\eta_2-2p(\frac{a^2}{b^2}+\frac{b^2}{a^2})\eta_2+(4p+1)\eta_2=0,
$$
which is equivalent to $a=b$, that is $M$ is minimal in
$\s^{4p+3}$.

(b) We obtain that $M$ is $(-4)$-biharmonic if and only if
$$
-\eta_2-2p(\frac{a^2}{b^2}+\frac{b^2}{a^2})\eta_2+(4p+1)\eta_2+4\eta_2=0,
$$
which holds if and only if $a^2=\frac{2p+1\pm\sqrt{2p+1}}{4p+2}$.

(c) We denote by $\dot A$ the shape operator of $M$ in
$\s^{2p+1}(a)\times\s^{2p+1}(b)$, $\dot A=\dot A_{\eta_1}$. We can
check that
\begin{equation}
\begin{cases}
\dot A(y_0^H)=0, \ \dot A(y_2^H)=-y_2^V, \
\dot A(y_3^H)=-y_3^V,\ldots, \ \dot A(y_{2p+1}^H)=-y_{2p+1}^V\\
\dot A(y_2^V)=-y_2^H, \ \dot A(y_3^V)=-y_3^H,\ldots, \ \dot
A(y_{2p+1}^V)=-y_{2p+1}^H
\end{cases}
\end{equation}
and therefore $\trace\dot A=0$, which means that $M$ is minimal in
$\s^{2p+1}(a)\times\s^{2p+1}(b)$.

(d) We first define
$$
\xi_3(x,y)=(\jr x,-\frac{a^2}{b^2}\jr
y)=(-\xi_1,\frac{a^2}{b^2}\xi_2),\quad\forall (x,y)\in
\s^{2p+1}(a)\times\s^{2p+1}(b).
$$
The vector field $\xi_3$ is a Killing vector field on
$\s^{2p+1}(a)\times\s^{2p+1}(b)$. We observe that
$\xi_{3/M}=\jr\eta_2$. Since $M$ is minimal in
$\s^{2p+1}(a)\times\s^{2p+1}(b)$, from Lemma \ref{lemmaexample},
we get $\di(\jr\eta_2)^\top=0$. Therefore $\pi(M)$ is biharmonic
in $\cp^n$ if and only if
$$
\tau_2(\is)-4\jr(\jr\tau(\is))^\top=0,
$$
which is not satisfied.
\end{proof}

\subsection{Circles products.} We shall recover a result of Zhang
(see~\cite{WZ}).

\noindent We denote by $\mathcal{T}$ the $(n+1)$-dimensional
Clifford torus
$$
\is:\mathcal{T}=\s^1(a_1)\times\cdots\times\s^1(a_{n+1})\to\s^{2n+1},
$$
where $a_1^2+\cdots +a_{n+1}^2=1$. The projection
$\bar{\mathcal{T}}=\pi(\mathcal{T})$ is a Lagrangian submanifold in
$\cp^n$ of parallel mean curvature vector field.

\begin{theorem}[\cite{WZ}]
\label{eq:Zhangtheorem} The Lagrangian submanifold
$\bar{\mathcal{T}}=\pi(\mathcal{T})$ of $\cp^n$ is proper-biharmonic
if and only if $\mathcal{T}$ is $(-4)$-biharmonic, that is
\begin{equation}
\begin{cases}
a_{k_{0}}^2\neq\frac{1}{n+1}\quad\textnormal{for some}\
k_{0}\in\{1,2,\ldots,n+1\}\\
d\,a_k-\frac{1}{a_k^3}=\frac{2}{a_k}(n+3)((n+1)a_k^2-1),\quad
k\in\{1,2,\ldots,n+1\}
\end{cases},
\end{equation}
where $d=\sum_{j=1}^{n+1}\frac{1}{a_j^2}$.
\end{theorem}

\begin{proof}
We denote a point $x\in \mathcal{T}$ by
$$
x=(x_1,\ldots,x_{n+1})=(x_1^1,x_1^2,\ldots,x_{n+1}^1,x_{n+1}^2),
$$
where we identify
$$
x_k=(x_k^1,x_k^2)=(0,0,\ldots,0,0,x_k^1,x_k^2,0,0,\ldots,0,0), \quad
k=1,\ldots,n+1.
$$
We define $\eta_k(x)=\frac{1}{a_k}x_k$ and $X_k=\jr\eta_k$,
$k=1,\ldots,n+1$, where
$$
\jr(x_1^1,x_1^2,\ldots,x_{n+1}^1,x_{n+1}^2)=(-x_1^2,x_1^1,\ldots,-x_{n+1}^2,x_{n+1}^1).
$$
The vector fields  $\{X_k\}$ form an orthonormal frame field of
$C(T\mathcal{T})$. It is easy to check that, at a point $x$,
$$
B(X_k,X_k)=-\frac{1}{a_k}\eta_k+x
$$
and for $k\neq j$:
$$
B(X_k,X_j)=0.
$$
Therefore $\tau(\is)=\sum_k((n+1)a_k-\frac{1}{a_k})\eta_k$, which
implies that $(\jr\tau(\is))^{\top}=\jr\tau(\is)$ and $\di(\jr
\tau(\is))=0$.

\noindent Since $\nabla^{\perp}\tau(\is)=0$ and
$A_{\tau(\is)}(X_k)=-((n+1)-\frac{1}{a_k^2})X_k$, by a
straightforward computation we get $\tau_2(\is)+4\tau(\is)=0$ if and
only if the desired relation is satisfied.
\end{proof}

\begin{remark}
Following~\cite{WZ}, for $n=2$, we obtain that $\bar{\mathcal{T}}$
is a proper-biharmonic Lagrangian surface in $\cp^2$ if and only
if $a_1^2=\frac{9\pm\sqrt{41}}{20}$ and
$a_2^2=a_3^2=\frac{11\mp\sqrt{41}}{40}$ (see also~\cite{TS}).
\end{remark}

\section{Biharmonic curves in $\cp^n$}\label{section-curvescpn}

Let $\bar{\gamma}:I\subset\r\to\cp^n$ be a curve parametrized by
arc-length. The curve $\gamma$ is called a {\it Frenet curve of
osculating order} $d$, $1\leq d\leq 2n$, if there exist $d$
orthonormal vector fields $\{\be_1=\bgamma',\ldots,\be_{d}\}$ along
$\bgamma$ such that
\begin{equation}
\begin{cases}
\nablab_{\be_{1}}\be_{1}=\bk_{1}\be_{2} \\
\nablab_{\be_{1}}\be_{i}=-\bk_{i-1}\be_{i-1} + \bk_{i}\be_{i+1},
\quad \forall i=2,\dots,d-1 \\
\nablab_{\be_{1}}\be_{d}=-\bk_{d-1}\be_{d-1}
\end{cases},
\end{equation}
where $\{\bk_{1},\bk_{2},\bk_{3},\ldots,\bk_{d-1}\}$ are positive
functions on $I$ called the {\it curvatures} of $\bgamma$ and
$\nablab$ denotes the Levi-Civita connection on $\cp^n$.

A Frenet curve of osculating order $d$ is called a {\it helix of
order $d$} if $\bk_i=\rm{constant}>0$ for $1\leq i\leq d-1$. A helix
of order $2$ is called a {\it circle}, and a helix of order $3$ is
simply called {\it helix}.

Following S.~Maeda and Y.~Ohnita \cite{SMYO}, we define the {\it
complex torsions} of the curve $\bgamma$ by $\bar{\tau}_{ij}=\langle
\be_i, \jc \be_j \rangle$, $1\leq i<j\leq d$. A helix of order $d$
is called a {\it holomorphic helix of order $d$} if all the complex
torsions are constant.

Using the Frenet equations, the bitension field of $\bgamma$ becomes
\begin{eqnarray}\label{eq:tau2frenet}
\tau_2(\bgamma)&=&-3\bk_1\bk_1'\be_1 + (\bk_1'' - \bk_1^3 - \bk_1
\bk_2^2 + \bk_1)\be_2 \\
&&+(2\bk_1'\bk_2 + \bk_1\bk_2')\be_3 + \bk_1\bk_2\bk_3\be_4 - 3\bk_1
\bar{\tau}_{12}\jc\be_1.\nonumber
\end{eqnarray}
In order to solve the biharmonic equation $\tau_2(\bar\gamma)=0$,
because of the last term in \eqref{eq:tau2frenet}, we must split our
study in three cases.

\subsection{Biharmonic curves with  $\mathbf{\bar{\tau}_{12}=\pm 1}$}

\noindent In this case $\jc\bar{E}_{2}=\pm E_{1}$ and, using the
Frenet equations of $\bgamma$, we obtain
$$
\jc(\nablab_{\bar{E}_{1}}\bar{E}_{1})=\pm\bk_{1}\bar{E}_{1}
=\bar{\nabla}_{\bar{E}_{1}}(\mp\bar{E}_{2})=\mp\nablab_{\bar{E}_{1}}\bar{E}_{2},
$$
so
$$
\nablab_{\bar{E}_{1}}\bar{E}_{2}=-\bk_{1}\bar{E}_{1}.
$$
Consequently, $\bk_{i}=0$, $i\geq 2$, and, from \eqref{eq:tau2frenet}, it follows

\begin{proposition}\label{p1s4}
A Frenet curve $\bgamma:I\subset\r\to\cp^{n}$ parametrized by
arc-length with $\bar{\tau}_{12}=\pm 1$ is proper-biharmonic if
and only if it is a circle with $\bk_{1}=2$.
\end{proposition}

Next, let us consider a curve $\bgamma:I\subset\r\to\cp^{n}$
parametrized by arc-length with $\bar{\tau}_{12}=\pm 1$, and denote
by $\gamma:I\subset\r\to\s^{2n+1}$ one of its horizontal lifts. We
shall characterize the biharmonicity of $\bgamma$ in terms of
$\gamma$.

We denote by $\dot\nabla$ the Levi-Civita connection on $\s^{2n+1}$.
We have $\gamma'=E_{1}=(\bar{E}_{1})^{H}$ and
$$
\dot\nabla_{E_{1}}E_{1}=(\bar{\nabla}_{\bar{E}_{1}}\bar{E}_{1})^{H}
=\bk_{1}{\bar{E}_{2}}^{H}=k_{1}E_{2},
$$
i.e. $k_{1}=\bk_{1}$ and
$E_{2}={\bar{E}_{2}}^{H}=\mp(\jc\bar{E}_{1})^{H}=\mp\jr E_{1}$. It
follows
\begin{eqnarray*}
\dot\nabla_{E_{1}}E_{2}&=&(\bar{\nabla}_{\bar{E}_{1}}\bar{E}_{2})^{H}+\langle
\dot\nabla_{E_{1}}E_{2},\xi\rangle\xi\\&=&-k_{1}E_{1}-\langle
E_{2},\dot\nabla_{E_{1}}\xi\rangle\xi\\&=&-k_{1}E_{1}\mp\langle
E_{2},E_{2}\rangle\xi\\&=&-k_{1}E_{1}\mp\xi
\end{eqnarray*}
and this means $k_{2}=1$ and $E_{3}=\mp\xi$. Then
$\dot\nabla_{E_{1}}E_{3}=\mp\dot\nabla_{E_{1}}\xi=-E_{2}$.

\noindent In conclusion $\gamma$ is a helix with $k_{1}=\bk_{1}$
and $k_{2}=1$.

\noindent Now, we have $\jr\tau(\gamma)=k_{1}\jr E_{2}=\pm
k_{1}E_{1}$, which is tangent to $\gamma$, and then
$$
\jr\{(\jr\tau(\gamma))^{\top}\}=\jr^{2}\tau(\gamma)=-\tau(\gamma).
$$

\noindent From
\begin{eqnarray*}
\Div\{(\jc\tau(\bar{\gamma}))^{\top}\}&=&\Div\{\bk_{1}\langle
\jc\bar{E}_{2},\bar{E}_{1}\rangle\bar{E}_{1}\}\\&=&\langle
\bar{\nabla}_{\bar{E}_{1}}(\bk_{1}\langle
\jc\bar{E}_{2},\bar{E}_{1}\rangle)\bar{E}_{1},\bar{E}_{1}\rangle\\&=&\bk'_{1}\langle
\jc\bar{E}_{2},\bar{E}_{1}\rangle+
\bk_{1}\langle\jc\bar{\nabla}_{\bar{E}_{1}}\bar{E}_{2},\bar{E}_{1}\rangle\\
&=&\pm\bk'_{1}=0,
\end{eqnarray*}
applying Remark~\ref{rem:div0} (iii),  we have

\begin{proposition}
A Frenet curve $\bgamma:I\subset\r\to\cp^{n}$ parametrized by
arc-length with $\bar{\tau}_{12}=\pm 1$ is proper-biharmonic if
and only if its horizontal lift $\gamma:I\subset\r\to\s^{2n+1}$ is
$(-4)$-biharmonic, i.e. $\gamma$ is a helix with $k_1=2$ and
$k_2=1$.
\end{proposition}

Moreover, we can obtain the explicit parametric equations of the
horizontal lifts of a proper-biharmonic Frenet curve
$\bgamma:I\to\cp^{n}$.

\begin{proposition}
Let $\bgamma:I\subset\r\to\cp^{n}$ be a proper-biharmonic Frenet
curve parametrized by arc-length with $\bar{\tau}_{12}=\pm 1$. Then
its horizontal lift $\gamma:I\subset\r\to\s^{2n+1}$ can be
parametrized in the Euclidean space $\r^{2n+2}$ by
\begin{eqnarray*}
\gamma(s)&=&\frac{\sqrt{2-\sqrt{2}}}{2}\cos((\sqrt{2}+1)s)e_{1}-
\frac{\sqrt{2-\sqrt{2}}}{2}\sin((\sqrt{2}+1)s)\jr e_{1}\\
\\&&+\frac{\sqrt{2+\sqrt{2}}}{2}\cos((\sqrt{2}-1)s)e_{3}+
\frac{\sqrt{2+\sqrt{2}}}{2}\sin((\sqrt{2}-1)s)\jr{e_{3}},
\end{eqnarray*}
where $e_1$ and $e_3$ are constant unit vectors in $\r^{2n+2}$ with
$e_3$ orthogonal to $e_1$ and $\jr e_1$.
\end{proposition}
\begin{proof} The curve $\gamma$ is a helix with the Frenet frame field
$\{E_{1}=\bar{E}_{1}^{H},E_{2}=\bar{E}_{2}^{H},E_3=\mp\xi\}$ and
with curvatures $k_{1}=\bk_1=2$ and $k_{2}=1$.

From the Weingarten equation of $\s^{2n+1}$ in $\r^{2n+2}$ and
Frenet equations we get
$$
\nablar_{E_{1}}E_{1}=\dot\nabla_{E_{1}}E_{1}-\langle
E_{1},E_{1}\rangle\gamma=k_{1}E_{2}-\gamma,
$$
$$
\nablar_{E_{1}}\nablar_{E_{1}}E_{1}=k_{1}\nablar_{E_{1}}E_{2}-E_{1}=
k_{1}(-k_{1}E_{1}\mp\xi)-E_{1}=-(k_{1}^{2}+1)E_{1}\mp k_{1}\xi
$$
and
\begin{eqnarray*}
\nablar_{E_{1}}\nablar_{E_{1}}\nablar_{E_{1}}E_{1}&=&
-(k_{1}^{2}+1)\nablar_{E_{1}}E_{1}\mp k_{1}\nablar_{E_{1}}\xi\\&=&
-(k_{1}^{2}+1)\nablar_{E_{1}}E_{1}-k_{1}E_{2}\\
&=&-6\gamma''-\gamma.
\end{eqnarray*}

\noindent Hence $\gamma$ is a solution of the differential equation
$$
\gamma^{iv}+6\gamma''+\gamma=0,
$$
whose general solution is
$$
\gamma(s)=\cos(As)c_{1}+\sin(As)c_{2}+\cos(Bs)c_{3}+\sin(Bs)c_{4},
$$
where $A,B=\sqrt{2}\pm 1$ and $\{c_{i}\}$ are constant vectors in
$\mathbb{E}^{2n+2}$.

\noindent As $\gamma$ satisfies
$$
\langle\gamma,\gamma\rangle=1,\quad
\langle\gamma',\gamma'\rangle=1,\quad
\langle\gamma,\gamma'\rangle=0,\quad
\langle\gamma',\gamma''\rangle=0,\quad
\langle\gamma'',\gamma''\rangle=1+\kappa_{1}^{2}=5,
$$
$$
\langle\gamma,\gamma''\rangle=-1,\quad
\langle\gamma',\gamma'''\rangle=-(1+\kappa_{1}^{2})=-5,\quad
\langle\gamma'',\gamma'''\rangle=0,
$$
$$
\langle\gamma,\gamma'''\rangle=0,\quad
\langle\gamma''',\gamma'''\rangle=7\kappa_{1}^{2}+1=29,
$$
and since, in $s=0$, we have $\gamma=c_{1}+c_{3}$,
$\gamma'=Ac_{2}+Bc_{4}$, $\gamma''=-A^{2}c_{1}-B^{2}c_{3}$,
$\gamma'''=-A^{3}c_{2}-B^{3}c_{4}$, we obtain
\begin{equation}\label{1.11}
c_{11}+2c_{13}+c_{33}=1
\end{equation}
\begin{equation}\label{1.21}
A^{2}c_{22}+2ABc_{24}+B^{2}c_{44}=1
\end{equation}
\begin{equation}\label{1.31}
Ac_{12}+Ac_{23}+Bc_{14}+Bc_{34}=0
\end{equation}
\begin{equation}\label{1.41}
A^{3}c_{12}+AB^{2}c_{23}+A^{2}Bc_{14}+B^{3}c_{34}=0
\end{equation}
\begin{equation}\label{1.51}
A^{4}c_{11}+2A^{2}B^{2}c_{13}+B^{4}c_{33}=5
\end{equation}
\begin{equation}\label{1.61}
A^{2}c_{11}+(A^{2}+B^{2})c_{13}+B^{2}c_{33}=1
\end{equation}
\begin{equation}\label{1.71}
A^{4}c_{22}+(AB^{3}+A^{3}B)c_{24}+B^{4}c_{44}=5
\end{equation}
\begin{equation}\label{1.81}
A^{5}c_{12}+A^{3}B^{2}c_{23}+A^{2}B^{3}c_{14}+B^{5}c_{34}=0
\end{equation}
\begin{equation}\label{1.91}
A^{3}c_{12}+A^{3}c_{23}+B^{3}c_{14}+B^{3}c_{34}=0
\end{equation}
\begin{equation}\label{1.101}
A^{6}c_{22}+2A^{3}B^{3}c_{24}+B^{6}c_{44}=29
\end{equation}
where $c_{ij}=\langle c_{i},c_{j}\rangle$. From (\ref{1.31}),
(\ref{1.41}), (\ref{1.81}) and (\ref{1.91}) it follows that
$$
c_{12}=c_{23}=c_{14}=c_{34}=0.
$$

\noindent The equations (\ref{1.11}), (\ref{1.51}) and
(\ref{1.61}) give
$$
c_{11}=\frac{1-B^{2}}{A^{2}-B^{2}},\quad c_{13}=0,\quad
c_{33}=\frac{A^{2}-1}{A^{2}-B^{2}}
$$
and, from (\ref{1.21}), (\ref{1.71}) and (\ref{1.101}) it follows
that
$$
c_{22}=\frac{1-B^{2}}{A^{2}-B^{2}},\quad c_{24}=0,\quad
c_{44}=\frac{A^{2}-1}{A^{2}-B^{2}}.
$$

\noindent Therefore, we obtain that $\{c_{i}\}$ are orthogonal
vectors in $\mathbb{E}^{2n+2}$ with $\vert c_{1}\vert=\vert
c_{2}\vert=\sqrt{\frac{1-B^{2}}{A^{2}-B^{2}}}$, $\vert
c_{3}\vert=\vert c_{4}\vert=\sqrt{\frac{A^{2}-1}{A^{2}-B^{2}}}$.

By using that $E_{1}=\gamma'\perp\xi$ and then that $\jr E_{2}=\pm
E_{1}$, we conclude.
\end{proof}

\begin{remark}
Under the flow-action of $\xi$, the $(-4)$-biharmonic curves
$\gamma$ induce the $(-4)$-biharmonic surfaces obtained in
Example~\ref{example}.
\end{remark}

\subsection{Biharmonic curves with  $\mathbf{\bar{\tau}_{12}=0}$}

\noindent From the expression \eqref{eq:tau2frenet} of the bitension
field of $\bgamma$ we obtain that $\bgamma$ is proper-biharmonic if
and only if
\begin{equation}
\begin{cases}
\bk_{1}=\cst>0,\ \ \bk_{2}=\cst\\
\bk_{1}^{2}+\bk_{2}^{2}=1\\ \bk_{2}\bk_{3}=0
\end{cases}.
\end{equation}

\begin{proposition}
\label{p2s4}
A Frenet curve $\bgamma:I\subset\r\to\cp^{n}$ parametrized by
arc-length with $\bar{\tau}_{12}=0$ is proper-biharmonic if and
only if either
\begin{itemize}
\item[(a)] $n = 2$ and $\bgamma$ is a circle with $\bk_1=1$,
\item[] or
\item[(b)] $n\geq 3$ and $\bgamma$ is a circle with $\bk_1=1$ or a helix with
$\bk_{1}^{2}+\bk_{2}^{2}=1$.
\end{itemize}
\end{proposition}

\begin{proof} We only have to prove the statements concerning the
dimension $n$.

First, since $\{\bar{E}_{1},\bar{E}_{2},\jc\bar{E}_{2}\}$ are
linearly independent, it follows that $n>1$.

Now, assume that $\bgamma$ is a Frenet curve of osculating order 3
such that $\jc\bar{E}_{2}\perp\bar{E}_{1}$. We have
\begin{equation}
\begin{cases}
\bar{E}_{1}=\bgamma'\\
\nablab_{\bar{E}_{1}}\bar{E}_{1}=\bk_{1}\bar{E}_{2}\\
\nablab_{\bar{E}_{1}}\bar{E}_{2}=-\bk_{1}\bar{E}_{1}+\bk_{2}\bar{E}_{3}\\
\nablab_{\bar{E}_{1}}\bar{E}_{3}=-\bk_{2}\bar{E}_{2}
\end{cases}.
\end{equation}
\noindent It is easy to see that, at an arbitrary point, the
system
$$
S_{1}=\{\bar
{E}_{1},\bar{E}_{2},\bar{E}_{3},\jc\bar{E}_{1},\jc\bar{E}_{2}\}
$$
consists of non-zero vectors which are orthogonal to each other, and
therefore $n\geq 3$.
\end{proof}

Next, we shall consider the horizontal lift
$\gamma:I\subset\r\to\s^{2n+1}$ of a curve
$\bgamma:I\subset\r\to\cp^{n}$ parametrized by arc-length with
$\bar{\tau}_{12}=0$. As in the previous case we have
$\gamma'=E_{1}=\bar{E}_{1}^{H}$, $E_{2}=\bar{E}_{2}^{H}$ and then
$\jr E_{2}\perp E_{1}$. This means $\jr(\tau(\gamma))\perp E_{1}$,
so $(\jr(\tau(\gamma)))^{\top}=0$. From Theorem
\ref{teo:reltautau} we obtain

\begin{proposition}
A Frenet curve $\bgamma:I\subset\r\to\cp^{n}$ parametrized by
arc-length with $\bar{\tau}_{12}=0$ is proper-biharmonic if and
only if its horizontal lift $\gamma:I\subset\r\to\s^{2n+1}$ is
proper-biharmonic.
\end{proposition}

The parametric equations of the proper-biharmonic Frenet curves in $\s^{2n+1}$
with $\jr E_{2}\perp E_{1}$ were obtained in \cite{DFCO}. Using
that result we can state

\begin{proposition}
Let $\bgamma:I\subset\r\to\cp^{n}$ be a proper-biharmonic Frenet
curve parametrized by arc-length with $\bar{\tau}_{12}=0$. Then
the
 horizontal lift $\gamma:I\subset\r\to\s^{2n+1}$ can be parametrized, in
the Euclidean space $\r^{2n+2}$,  either by
$$
\gamma(s)=\frac{1}{\sqrt{2}}\cos(\sqrt{2}s)e_{1}+
\frac{1}{\sqrt{2}}\sin(\sqrt{2}s)e_{2}+\frac{1}{\sqrt{2}}e_{3},
$$
where $\{e_{i},\jr e_{j}\}_{i,j=1}^{3}$ are constant unit vectors
orthogonal to each other, or by
$$
\begin{array}{lll}
\gamma(s)&=&\frac{1}{\sqrt{2}}\cos(\sqrt{1+\kappa_{1}}s)e_{1}
+\frac{1}{\sqrt{2}}\sin(\sqrt{1+\kappa_{1}}s)e_{2} \\ \\
&&+\frac{1}{\sqrt{2}}\cos(\sqrt{1-\kappa_{1}}s)e_{3}
+\frac{1}{\sqrt{2}}\sin(\sqrt{1-\kappa_{1}}s)e_{4},
\end{array}
$$
where $\kappa_{1}\in(0,1)$, and $\{e_{i},\jr e_{j}\}_{i,j=1}^{4}$
are constant unit vectors orthogonal to each other.
\end{proposition}

\subsection{Biharmonic curves with  $\mathbf{\bar{\tau}_{12}}$ different from
$\mathbf{0}$, $\mathbf{1}$ or $\mathbf{-1}$}

\noindent Assume that $\bgamma$ is a proper-biharmonic Frenet
curve of osculating order $d$ such that $\bar{\tau}_{12}$ is
different from $0,1$ or $-1$.

First, we shall prove that $d\geq 4$.

\noindent Assume that $d=2$. From the biharmonic equation
$\tau_{2}(\bgamma)=0$ we have $\bk_{1}=\cst>0$ and then
$(-\bk_{1}^{3}+\bk_{1})\bar{E}_{2}-3\bk_{1}\bar{\tau}_{12}\jc\bar{E}_{1}=0$.
It follows that $\bar{E}_{2}$ is parallel to $\jc\bar{E}_{1}$,
i.e. $\bar{\tau}_{12}^2=1$.

\noindent Now, if $d=3$, from the biharmonic equation of
$\bgamma$, we obtain again $\bk_{1}=\cst>0$ and then
\begin{equation}\label{eq:1c3}
(-\bk_{1}^{2}-\bk_{2}^{2}+1)\bar{E}_{2}+\bk_{2}'\bar{E}_{3}-3\bar{\tau}_{12}\jc\bar{E}_{1}=0.
\end{equation}

\noindent Next, differentiating
$-\bar{\tau}_{12}(s)=\langle\bar{E}_{2},\jc\bar{E}_{1}\rangle$, we
obtain
\begin{eqnarray*}
-\bar{\tau}_{12}'(s)&=&\langle\nablab_{\bar{E}_{1}}\bar{E}_{2},\jc\bar{E}_{1}\rangle+
\langle\bar{E}_{2},\nablab_{\bar{E}_{1}}\jc\bar{E}_{1}\rangle
=\langle\nablab_{\bar{E}_{1}}\bar{E}_{2},\jc\bar{E}_{1})+
\langle\bar{E}_{2},\bk_{1}\jc\bar{E}_{2})\\
&=&\langle\nablab_{\bar{E}_{1}}\bar{E}_{2},\jc\bar{E}_{1}\rangle=\langle
-\bk_{1}\bar{E}_{1}+\bk_{2}\bar{E}_{3},\jc\bar{E}_{1}\rangle\\
&=&\bk_{2}\langle \bar{E}_{3},\jc\bar{E}_{1}\rangle.
\end{eqnarray*}
Hence, taking the inner product with $\bk_2\bar{E}_3$ in
\eqref{eq:1c3}, we get
$\bk_{2}'\bk_{2}+3\bar{\tau}_{12}\bar{\tau}_{12}'=0$ and so
$\bk_{2}^{2}=-3\bar{\tau}_{12}^{2}+\omega_{0}$, where
$\omega_{0}=\cst$. Using \eqref{eq:1c3} it results that
$\bk_{1}^{2}=1-\omega_{0}+6\bar{\tau}_{12}^{2}$. Therefore
$f=\cst$ and $\bk_{2}=\cst$. Finally, \eqref{eq:1c3} becomes
$(-\bk_{1}^{2}-\bk_{2}^{2}+1)\bar{E}_{2}-3\bar{\tau}_{12}\jc\bar{E}_{1}=0$,
which means that $\bar{E}_{2}$ is parallel to $\jc\bar{E}_{1}$.

\noindent We have proved the following

\begin{proposition}\label{prop:order4}
Let $\bgamma$ be a proper-biharmonic Frenet curve in $\cp^{n}$ of
osculating order $d$, $1\leq d\leq 2n$, with $\bar{\tau}_{12}$
different from $0$, $1$ or $-1$. Then $d\geq 4$.
\end{proposition}

Next we shall prove that for a proper-biharmonic Frenet curve in
$\cp^{n}$, $\bar{\tau}_{12}$ and $\bk_{1}$ are constants whatever
the osculating order of $\bgamma$ is.

\noindent We have seen that $-\bar{\tau}_{12}'(s)=\bk_{2}\langle
\bar{E}_{3},\jc\bar{E}_{1}\rangle$. If $\tau_{2}(\bgamma)=0$ we
have
$\jc\bar{E}_{1}=\langle\jc\bar{E}_{1},\bar{E}_{2}\rangle\bar{E}_{2}+
\langle\jc\bar{E}_{1},\bar{E}_{3}\rangle\bar{E}_{3}+
\langle\jc\bar{E}_{1},\bar{E}_{4}\rangle\bar{E}_{4}$ and
\begin{equation}
\label{sys-bi-curve-cpn}
\begin{cases}
\bk_{1}=\cst>0\\
\bk_{1}^{2}+\bk_{2}^{2}=1+3\bar{\tau}_{12}^{2}\\
\bk_2\bk_{2}'=-3\bar{\tau}_{12}\bar{\tau}_{12}'\\
\bk_{2}\bk_{3}=3\bar{\tau}_{12}\langle\jc\bar{E}_{1},\bar{E}_{4}\rangle
\end{cases}.
\end{equation}

\noindent From the third equation of \eqref{sys-bi-curve-cpn},  we get
$$
\bk_{2}^{2}=-3\bar{\tau}_{12}^{2}+\omega_{0},
$$
where $\omega_{0}=\cst$. Replacing in the second equation of \eqref{sys-bi-curve-cpn} it
follows that
$$
\bk_{1}^{2}=1+6\bar{\tau}_{12}-\omega_{0},
$$
which implies $\bar{\tau}_{12}=\cst$, and therefore,
$\bk_{2}=\cst>0$. From  $-\bar{\tau}_{12}'(s)=\bk_{2}\langle
\bar{E}_{3},\jc\bar{E}_{1}\rangle$, we have
$\langle\jc\bar{E}_{1},\bar{E}_{3}\rangle=0$ and then
$\jc\bar{E}_{1}=f\bar{E}_{2}+\langle\jc\bar{E}_{1},\bar{E}_{4}\rangle\bar{E}_{4}$.
It follows that there exists an unique constant $\alpha_{0}\in
(0,2\pi)\setminus\{\frac{\pi}{2},\pi,\frac{3\pi}{2}\}$ such that
$-\bar{\tau}_{12}=\cos\alpha_{0}$ and
$\langle\jc\bar{E}_{1},\bar{E}_{4}\rangle=
\sin\alpha_{0}=\frac{\bk_{2}\bk_{3}}{3\bar{\tau}_{12}}$.

\noindent We can summarise in

\begin{proposition}\label{prop:caracterizareordin4}
A Frenet curve $\bgamma:I\subset\r\to\cp^{n}$, $n\geq 2$,
parametrized by arc-length with $\bar{\tau}_{12}$ different from
$0$, $1$ or $-1$ is proper-biharmonic if and only if
$\jc\bar{E}_{1}=\cos\alpha_{0}\bar{E}_{2}+\sin\alpha_{0}\bar{E}_{4}$
and
\begin{equation}
\label{eq:sistemcurburi}
\begin{cases}
\bk_{1},\bk_2,\bk_3=\cst>0\\
\bk_{1}^{2}+\bk_{2}^{2}=1+3\cos^{2}\alpha_{0}\\
\bk_{2}\bk_{3}=-\frac{3}{2}\sin(2\alpha_{0})\\
\bar{\tau}_{12}=-\cos\alpha_{0}
\end{cases},
\end{equation}
where $\alpha_{0}\in (\frac{\pi}{2},\pi)\cup(\frac{3\pi}{2},2\pi)$
is a constant.
\end{proposition}

We end this section classifying the proper-biharmonic curves in
$\cp^n$ of osculating order $d\leq 4$. First,

\begin{proposition}\label{pro:tauhelix}
Let $\bgamma$ be a proper-biharmonic Frenet curve in $\cp^n$ of
osculating order $d<4$. Then $\bgamma$ is one of the following:
a holomorphic circle of curvature $\bk_1=2$, a holomorphic circle of curvature $\bk_1=1$,
or a holomorphic helix with $\bk_1^2+\bk_2^2=1$.
\end{proposition}

\begin{proof}
Let $\bgamma$ be a proper-biharmonic Frenet curve of osculating
order $d<4$. Then, from Proposition~\ref{prop:order4},
$\bar{\tau}_{12}=\pm 1$ or $\bar{\tau}_{12}=0$. If
$\bar{\tau}_{12}=\pm 1$, from Proposition~\ref{p1s4},
 $\bgamma$ is a circle of curvature
$\bk_1=2$. If $\bar{\tau}_{12}=0$ then we know that $\bgamma$ is
either a holomorphic circle of curvature $\bk_1=1$ or a helix. We
now prove that it is a holomorphic helix. For this we need to
prove that the complex torsions $\bar{\tau}_{13},\bar{\tau}_{23}$
are constant.
\begin{eqnarray*}
\bar{\tau}_{13}&=&\langle
\be_1,\jc\be_3\rangle=-\frac{1}{\bk_2}\langle
\nablab_{\be_1}\be_2,\jc\be_1\rangle =\frac{1}{\bk_2}\langle
\be_2,\nablab_{\be_1}\jc\be_1\rangle \\
&=&\frac{\bk_1}{\bk_2}\langle \be_2,\jc\be_2\rangle=0.
\end{eqnarray*}

\noindent Now, using that for a Frenet curve of osculating order
$3$ we have
$\bk_1\bar{\tau}_{23}=\bar{\tau}_{13}'+\bk_2\bar{\tau}_{12}$, we
see that also $\bar{\tau}_{23}$ is constant.
\end{proof}

When the biharmonic curve is of osculating order $4$,
system~\eqref{eq:sistemcurburi} has four solutions.

\begin{proposition}\label{prop:eliciolomorfeordin4}
Let $\bgamma$ be a proper-biharmonic Frenet curve in $\cp^n$ of
osculating order $d=4$. Then $\bgamma$ is a holomorphic helix.
Moreover, depending on the value of $\bar{\tau}_{12}=-\cos\a0$, we
have
\begin{itemize}
\item[(a)] If $\bar{\tau}_{12}>0$, then the curvatures of
$\bgamma$ are given by
\begin{equation}
\begin{cases}
\bk_2=\frac{\sin\alpha_0}{\sqrt{2}}\sqrt{1-3\cos^2\alpha_0\pm\sqrt{9\cos^4\alpha_0-42\cos^2\alpha_0+1}}\\
\bk_3=-\frac{3}{2\bk_2}\sin(2\alpha_0)\\
\bk_1=-\frac{1}{\sin\alpha_0}(\bk_2\cos\alpha_0-\bk_3\sin\alpha_0)
\end{cases}
\end{equation}
and
$$
\bar{\tau}_{34}=-\bar{\tau}_{12}=\cos\alpha_0, \quad
\bar{\tau}_{14}=-\bar{\tau}_{23}=-\sin\alpha_0 \quad \hbox{and}
\quad \bar{\tau}_{13}=\bar{\tau}_{24}=0,
$$
where
$\alpha_0\in(\frac{\pi}{2},\arccos(-\frac{2-\sqrt{3}}{\sqrt{2}}))$.
\item[(b)] If $\bar{\tau}_{12}<0$, then the curvatures of
$\bgamma$ are given by
\begin{equation}
\begin{cases}
\bk_2=-\frac{\sin\alpha_0}{\sqrt{2}}\sqrt{1-3\cos^2\alpha_0\pm\sqrt{9\cos^4\alpha_0-42\cos^2\alpha_0+1}}\\
\bk_3=-\frac{3}{2\bk_2}\sin(2\alpha_0)\\
\bk_1=-\frac{1}{\sin\alpha_0}(\bk_2\cos\alpha_0-\bk_3\sin\alpha_0)
\end{cases}
\end{equation}
and
$$
\bar{\tau}_{34}=-\bar{\tau}_{12}=\cos\alpha_0, \quad
\bar{\tau}_{14}=-\bar{\tau}_{23}=-\sin\alpha_0 \quad \hbox{and}
\quad \bar{\tau}_{13}=\bar{\tau}_{24}=0,
$$
where
$\alpha_0\in(\frac{3\pi}{2},\pi+\arccos(-\frac{2-\sqrt{3}}{\sqrt{2}}))$.
\end{itemize}
\end{proposition}

\begin{proof}
Let $\bgamma$ be a proper-biharmonic Frenet curve in $\cp^n$ of
osculating order $d=4$. Then $\bar{\tau}_{12}=-\cos\alpha_0$ is
different from $0$, $1$ or $-1$, and
$\jc\bar{E}_{1}=\cos\alpha_{0}\bar{E}_{2}+\sin\alpha_{0}\bar{E}_{4}$.
Then it results that
$$
\bar{\tau}_{12}=-\cos\alpha_0, \quad \bar{\tau}_{13}=0, \quad
\bar{\tau}_{14}=-\sin\alpha_0, \quad \hbox{and} \quad
\bar{\tau}_{24}=0.
$$
In order to prove that $\bar{\tau}_{23}$ is constant we
differentiate the expression of $\jc\bar{E}_{1}$ and using the
Frenet equations we obtain
\begin{eqnarray*}
\nablab_{\be_1}\jc\be_1&=&\cos\alpha_0\nablab_{\be_1}\be_2
+\sin\alpha_0\nablab_{\be_1}\be_4 \\
&=&-\bk_1\cos\alpha_0\be_1+(\bk_2\cos\alpha_0-\bk_3\sin\alpha_0)\be_3.
\end{eqnarray*}
On the other hand, $\nablab_{\be_1}\jc\be_1=\bk_1\jc\be_2$ and
therefore we have
\begin{equation}\label{eq:relatiecurburi}
\bk_1\jc\be_2=-\bk_1\cos\alpha_0\be_1+(\bk_2\cos\alpha_0-\bk_3\sin\alpha_0)\be_3.
\end{equation}
We take the inner product of ~\eqref{eq:relatiecurburi} with
$\be_3$, $\jc\be_2$ and $\jc\be_4$, respectively, and we get
\begin{equation}\label{eq:primarelatietau23}
\bk_1\bar{\tau}_{23}=-(\bk_2\cos\alpha_0-\bk_3\sin\alpha_0),
\end{equation}
\begin{equation}\label{eq:adouarelatietau23}
\bk_1\sin^2\alpha_0=-(\bk_2\cos\alpha_0-\bk_3\sin\alpha_0)\bar{\tau}_{23},
\end{equation}
\begin{equation}\label{eq:relatietau34}
0=\bk_1\cos\alpha_0\sin\alpha_0+(\bk_2\cos\alpha_0-\bk_3\sin\alpha_0)\bar{\tau}_{34}.
\end{equation}
From ~\eqref{eq:primarelatietau23} and ~\eqref{eq:adouarelatietau23}
we obtain
\begin{equation}\label{eq:valuarekapa2}
\bk^2_1\sin^2\alpha_0=(\bk_2\cos\alpha_0-\bk_3\sin\alpha_0)^2
\end{equation}
and $\tau^2_{23}=\sin^2\alpha_0$. From $\tau^2_{23}=\sin^2\alpha_0$,
~\eqref{eq:primarelatietau23} and $\alpha_{0}\in
(\frac{\pi}{2},\pi)\cup(\frac{3\pi}{2},2\pi)$, one obtains
$$
\bar{\tau}_{23}=\sin\alpha_0.
$$
From $\bar{\tau}_{23}=\sin\alpha_0$, ~\eqref{eq:primarelatietau23}
and ~\eqref{eq:relatietau34} we get
$$
\bar{\tau}_{34}=\cos\alpha_0.
$$

\noindent Finally, from Proposition~\ref{prop:caracterizareordin4}
and~\eqref{eq:valuarekapa2} we obtain
$$
\bk_2^4+\bk_2^2\sin^2\a0(3\cos^2\a0-1)+9\sin^4\a0\cos^2\a0=0.
$$
The latter equation has either  the solutions
$$
\bk_2=\frac{\sin\alpha_0}{\sqrt{2}}
\sqrt{1-3\cos^2\alpha_0\pm\sqrt{9\cos^4\alpha_0-42\cos^2\alpha_0+1}}
$$
provided that
$\alpha_0\in(\frac{\pi}{2},\arccos(-\frac{2-\sqrt{3}}{\sqrt{2}}))$,
or the solutions
$$
\bk_2=-\frac{\sin\alpha_0}{\sqrt{2}}
\sqrt{1-3\cos^2\alpha_0\pm\sqrt{9\cos^4\alpha_0-42\cos^2\alpha_0+1}}
$$
provided that
$\alpha_0\in(\frac{3\pi}{2},\pi+\arccos(-\frac{2-\sqrt{3}}{\sqrt{2}}))$.
Note that in both cases $\bk_2^2\in(0,4)$, thus all solutions for
$\bk_2$ are compatible with $ \bk_1^2+\bk_2^2=1+3\cos^2\alpha_0$.
\end{proof}

\begin{corollary}
Any proper-biharmonic Frenet curve in $\cp^2$ is a holomorphic
circle or a holomorphic helix of order $4$.
\end{corollary}

\begin{remark}
The existence of biharmonic curves of osculating order $d\geq 4$ is
an open problem (the case $d=4$ and $n=2$ will be solved in the next
section). We note that there is no curve (not necessarily
biharmonic) of order $d=5$ in $\cp^n$ such that
$\jc\bar{E}_{1}=\cos\alpha_{0}\bar{E}_{2}+\sin\alpha_{0}\bar{E}_{4}$,
where $\alpha_{0}\in (0,2\pi)\setminus\{\pi\}$.
\end{remark}

\section{Biharmonic curves in $\cp^2$}
In this section we give the complete classification of all
proper-biharmonic Frenet curves in $\cp^2$. From the previous
section, we only have to classify the
proper-biharmonic Frenet curves of osculating order $4$.

In the proof of Proposition~\ref{prop:eliciolomorfeordin4} we have
seen that
$$
\bar{\tau}_{34}=-\bar{\tau}_{12}=\cos\alpha_0, \quad
\bar{\tau}_{14}=-\bar{\tau}_{23}=-\sin\alpha_0 \quad \hbox{and}
\quad \bar{\tau}_{13}=\bar{\tau}_{24}=0,
$$
and
$$
\bk_1\sin\alpha_0=-(\bk_2\cos\alpha_0-\bk_3\sin\alpha_0),
$$
which implies that
$\bk_1-\bk_3=-\bk_2\frac{\cos\alpha_0}{\sin\alpha_0}>0$.

\noindent Moreover, if
$\alpha_0\in(\frac{\pi}{2},\arccos(-\frac{2-\sqrt{3}}{\sqrt{2}}))$,
then
$$
\frac{\bk_1-\bk_3}{\sqrt{\bk_2^2 +
(\bk_1-\bk_3)^2}}=-\cos\alpha_0=\bar{\tau}_{12}, \quad
\frac{\bk_2}{\sqrt{\bk_2^2 +
(\bk_1-\bk_3)^2}}=\sin\alpha_0=\bar{\tau}_{23},
$$
and, if
$\alpha_0\in(\frac{3\pi}{2},\pi+\arccos(-\frac{2-\sqrt{3}}{\sqrt{2}}))$,
then
$$
\frac{\bk_1-\bk_3}{\sqrt{\bk_2^2 +
(\bk_1-\bk_3)^2}}=\cos\alpha_0=-\bar{\tau}_{12}, \quad
\frac{\bk_2}{\sqrt{\bk_2^2 +
(\bk_1-\bk_3)^2}}=-\sin\alpha_0=-\bar{\tau}_{23}.
$$

In order to conclude, we briefly recall a result of S.~Maeda and
T.~Adachi.

\noindent In \cite{SMTA}, they showed that for
given positive constants $\bk_1,\bk_2$ and $\bk_3$, there exist four
equivalence classes of holomorphic helices of order $4$ in $\cp^2$
with curvatures $\bk_1,\bk_2$ and $\bk_3$ with respect to
holomorphic isometries of $\cp^2$. The four classes are defined by
certain relations on the complex torsions and they are: when
$\bk_1\neq\bk_3 $
$$
\begin{array}{|c|lll|}
\hline
&&\bk_1\neq\bk_3   &  \\
\hline
I_1& \bar{\tau}_{12}=\bar{\tau}_{34}=\mu&\bar{\tau}_{23}=\bar{\tau}_{14}=\bk_2\mu/(\bk_1+\bk_3)&\bar{\tau}_{13}=\bar{\tau}_{24}=0\\
\hline
I_2& \bar{\tau}_{12}=\bar{\tau}_{34}=-\mu&\bar{\tau}_{23}=\bar{\tau}_{14}=-\bk_2\mu/(\bk_1+\bk_3)&\bar{\tau}_{13}=\bar{\tau}_{24}=0\\
\hline
I_3& \bar{\tau}_{12}=-\bar{\tau}_{34}=\nu&\bar{\tau}_{23}=-\bar{\tau}_{14}=\bk_2\nu/(\bk_1-\bk_3)&\bar{\tau}_{13}=\bar{\tau}_{24}=0\\
\hline
I_4& \bar{\tau}_{12}=-\bar{\tau}_{34}=-\nu&\bar{\tau}_{23}=-\bar{\tau}_{14}=-\bk_2\nu/(\bk_1-\bk_3)&\bar{\tau}_{13}=\bar{\tau}_{24}=0\\
\hline
\end{array}
$$
where
$$
\begin{cases}
\mu=\dfrac{\bk_1+\bk_3}{\sqrt{\bk_2^2+(\bk_1+\bk_3)^2}}\\
\nu=\dfrac{\bk_1-\bk_3}{\sqrt{\bk_2^2+(\bk_1-\bk_3)^2}}
\end{cases},
$$
and when $\bk_1=\bk_3$ the classes $I_3$ and $I_4$ are substituted
by
$$
\begin{array}{|c|ll|}
\hline
&\bk_1=\bk_3   &  \\
\hline
I_3'& \bar{\tau}_{12}=\bar{\tau}_{34}=\bar{\tau}_{13}=\bar{\tau}_{24}=0 & \bar{\tau}_{23}=-\bar{\tau}_{14}=1\\
\hline
I_4'& \bar{\tau}_{12}=\bar{\tau}_{34}=\bar{\tau}_{13}=\bar{\tau}_{24}=0 & \bar{\tau}_{23}=-\bar{\tau}_{14}=-1\\
\hline
\end{array}
$$

Using Maeda-Adachi classification, we can conclude

\begin{theorem}
\label{eq:curvescp2} Let $\bgamma$ be a  proper-biharmonic Frenet
curve in $\cp^2$ of osculating order $4$. Then $\bgamma$  is a
holomorphic helix of order $4$ of class $I_3$ or $I_4$ according to
the following table
$$
\begin{array}{|lllll|}
\hline
I_3& \rm{if} &\bar{\tau}_{12}<0 & \rm{and} &\bar{\tau}_{23}<0\\
\hline
I_4&  \rm{if}  &\bar{\tau}_{12}>0 &\rm{and} & \bar{\tau}_{23}>0\\
\hline
\end{array}
$$

\noindent Conversely,
\begin{itemize}
\item[(a)] For any
$\alpha_0\in(\frac{\pi}{2},\arccos(-\frac{2-\sqrt{3}}{\sqrt{2}}))$
there exist two proper-biharmonic holomorphic helices of order $4$
of class $I_3$ with
\begin{equation}
\begin{cases}
\bk_2=\frac{\sin\alpha_0}{\sqrt{2}}\sqrt{1-3\cos^2\alpha_0\pm\sqrt{9\cos^4\alpha_0-42\cos^2\alpha_0+1}}\\
\bk_3=-\frac{3}{2\bk_2}\sin(2\alpha_0)\\
\bk_1=-\frac{1}{\sin\alpha_0}(\bk_2\cos\alpha_0-\bk_3\sin\alpha_0)
\end{cases}.
\end{equation}
\item[(b)] For any
$\alpha_0\in(\frac{3\pi}{2},\pi+\arccos(-\frac{2-\sqrt{3}}{\sqrt{2}}))$
there exist two proper-biharmonic  holomorphic helices of order $4$
of class $I_4$ with
\begin{equation}
\begin{cases}
\bk_2=-\frac{\sin\alpha_0}{\sqrt{2}}\sqrt{1-3\cos^2\alpha_0\pm\sqrt{9\cos^4\alpha_0-42\cos^2\alpha_0+1}}\\
\bk_3=-\frac{3}{2\bk_2}\sin(2\alpha_0)\\
\bk_1=-\frac{1}{\sin\alpha_0}(\bk_2\cos\alpha_0-\bk_3\sin\alpha_0)
\end{cases}.
\end{equation}
\end{itemize}
\end{theorem}

\noindent {\bf Acknowledgements.} This work started in November 2007
during a visit to Tohoku University by the second and last authors.
They wish to thank Professor Hajime Urakawa for his kindness and
hospitality. The last author also thanks INdAM for a three-week
grant at the University of Cagliary, January 2008.

\end{document}